\documentclass{amsart}
\usepackage[utf8]{inputenc}
\usepackage{amssymb}
\usepackage[leqno]{amsmath}
\usepackage{amsthm}
\usepackage{latexsym}
\usepackage{enumitem}
\usepackage{graphicx}
\usepackage{bigints}
\usepackage{esint}
\usepackage{ifthen}
\usepackage{tikz}
\usepackage{stmaryrd}
\usepackage{subfig}
\usepackage{todonotes}

\theoremstyle{plain}
\newtheorem{theorem}{Theorem}

\newtheorem{lemma}[theorem]{Lemma}
\newtheorem{proposition}[theorem]{Proposition}
\newtheorem{definition}[theorem]{Definition}

\theoremstyle{remark}
\newtheorem{remark}[theorem]{Remark}

\numberwithin{theorem}{section}
\numberwithin{equation}{section}
\numberwithin{figure}{section}
\numberwithin{table}{section}

\newcommand{\N}{\mathbb{N}}
\newcommand{\R}{\mathbb{R}}
\newcommand{\Cleq}{\lesssim}

\newcommand{\Supp}{\mathrm{supp}}

\newcommand{\Domain}{\Omega}
\newcommand{\Dim}{d}
\newcommand{\Normal}{n}

\newcommand{\Leb}[1]{L^2(\ifthenelse{\equal{#1}{}}{\Domain} {#1})}
\newcommand{\LebH}[1]{L^2_0(\ifthenelse{\equal{#1}{}}{\Domain} {#1})}
\newcommand{\Sob}[1]{H^1(\ifthenelse{\equal{#1}{}}{\Domain}{#1})}
\newcommand{\SobH}[1]{H^1_0(\ifthenelse{\equal{#1}{}}{\Domain}{#1})}
\newcommand{\SobD}[1]{H^{-1}(\ifthenelse{\equal{#1}{}}{\Domain}{#1})}
\newcommand{\Hdiv}[1]{H_{\mathrm{div}}(\ifthenelse{\equal{#1}{}}{\Domain}{#1})}

\DeclareMathOperator{\Grad}{\nabla}
\DeclareMathOperator{\GradM}{\Grad_\Mesh}

\DeclareMathOperator{\Div}{div}   
\DeclareMathOperator{\DivM}{\Div_\Mesh}
\DeclareMathOperator{\Divdisc}{\underline{\Div}_\textit{h}}
\DeclareMathOperator{\Lapl}{\Delta}

\DeclareMathOperator{\Curl}{curl}

\newcommand{\SPvel}{\SobH{}^\Dim}
\newcommand{\SPpres}{\LebH{}}
\newcommand{\SPker}{Z}
\newcommand{\SPveldisc}{V_h}
\newcommand{\SPpresdisc}{Q_h}
\newcommand{\SPkerdisc}{Z_h}

\newcommand{\Formpres}{b}
\newcommand{\Formveldisc}{a_h}
\newcommand{\Formpresdisc}{b_h}

\newcommand{\NormSemi}[1]{\left \lvert #1 \right \rvert}          
\newcommand{\Norm}[1]{\| #1 \|}                     
\newcommand{\Normh}[2][h]{\Norm{ #2}_{#1}}
\newcommand{\NormLeb}[2]{\Norm{#1}_{\Leb{#2}}}
\newcommand{\Normtr}[1]{ \lvert\!\lvert\!\lvert{#1} \rvert\!\rvert\!\rvert}

\newcommand{\Mesh}{\mathcal{M}}
\newcommand{\Submesh}{\Mesh}
\newcommand{\Nodes}[1]{\mathcal{V}_{#1}}
\newcommand{\Faces}[1]{\mathcal{F}_{#1}}
\newcommand{\Jump}[1]{\llbracket #1 \rrbracket}

\newcommand{\Refel}{\mathrm{ref}}

\newcommand{\Degree}{\ell}
\newcommand{\Poly}[1]{\mathbb{P}_{#1}}
\newcommand{\Polypiec}[2]{S_{#1}^{#2}}
\newcommand{\Polyavg}[2]{\widehat{S}_{#1}^{#2}}
\newcommand{\Polybnd}[2]{\mathring{S}_{#1}^{#2}}
\newcommand{\CR}[1]{\mathring{CR}_{#1}}

\newcommand{\Pl}{{ub}}
\newcommand{\Cr}{{cr}}
\newcommand{\Ht}{{ht}}

\newcommand{\Smt}{E_h}
\newcommand{\Smtavg}{A_h}
\newcommand{\Smtavgloc}{A}
\newcommand{\Smtfac}{B_h}

\newcommand{\Smtm}{M_h}
\newcommand{\SmtLag}{I}

\newcommand{\Id}{\mathrm{Id}}
\newcommand{\Ritz}{\Pi}

\newcommand{\Rightinv}{R}

\newcommand{\Cqo}{C_{\mathrm{qo}}}
\newcommand{\Cqopr}{\mathrm{C_{\mathrm{qopr}}}}

\newcommand{\Visc}{\mu}

\title[Quasi-optimal and pressure robust discretizations of Stokes]{Quasi-optimal and pressure robust\\ discretizations of the Stokes equations by\\new augmented Lagrangian formulations}

\author[C.~Kreuzer]{Christian Kreuzer}
\address{TU Dortmund \\ Fakult{\"a}t f{\"u}r Mathematik \\ D-44221 Dortmund \\ Germany}
\email{christian.kreuzer@tu-dortmund.de}
\author[P.~Zanotti]{Pietro Zanotti}
\address{TU Dortmund \\ Fakult{\"a}t f{\"u}r Mathematik \\ D-44221 Dortmund \\ Germany}
\email{zanottipie@gmail.com}

\begin{document}

\begin{abstract}
We approximate the solution of the stationary Stokes equations with
various conforming and nonconforming inf-sup stable pairs of finite
element spaces on simplicial meshes. Based on each pair, we design a
discretization that is quasi-optimal and pressure robust, in the sense
that the velocity $H^1$-error is proportional to the best $H^1$-error
to the analytical velocity. This shows that such a
property can be achieved without using conforming and divergence-free
pairs. We bound also the pressure $L^2$-error, only in terms of the
best approximation errors to the analytical velocity and the
analytical pressure. Our construction can be summarized as follows. First, a linear operator acts on discrete velocity test
functions, before the application of the load functional, and maps the
discrete kernel into the analytical one. Second, in order to enforce consistency, we employ a new augmented Lagrangian formulation, inspired by Discontinuous Galerkin methods.   
\end{abstract}
\maketitle

\section{Introduction}
\label{S:introduction}

We consider the discretization of the stationary Stokes equations
\begin{equation}
\label{Stokes-strong}
-\Visc \Lapl u + \Grad p = f
\quad \text{and} \quad 
\Div u = 0
\quad \text{in } \Domain,
\qquad
u = 0 
\quad \text{on } \partial \Domain 
\end{equation}
with viscosity $\Visc > 0$, in a bounded domain $\Domain \subseteq
\R^\Dim$, $\Dim \in \{2,3\}$. According to the classical approach of
Brezzi~\cite{Brezzi:74}, we approximate the analytical velocity $u$
and the analytical pressure $p$ by means of discrete spaces
$\SPveldisc$ and $\SPpresdisc$, which are required to fulfill the so-called inf-sup condition. We additionally assume that $\SPveldisc$
and $\SPpresdisc$ are finite element spaces on a simplicial mesh of
$\Domain$. 

To motivate our work, let us focus on the velocity $H^1$-error, i.e. the error between $u$ and the discrete velocity $u_h$, measured in the
$H^1$-norm.  We refer to~\cite[Chapter~5]{Boffi:Brezzi:Fortin.13} for the proof of the results listed hereafter. The C\'{e}a's-type quasi-optimal estimate    
\begin{equation}
\label{intro:div-free-est}
\NormLeb{\Grad(u-u_h)}{}
\leq c
\inf_{w_h \in \SPveldisc} \NormLeb{\Grad(u-w_h)}{}
\end{equation}
is well-known for standard discretizations (see~\eqref{Stokes-disc} and~\eqref{conforing-div-free-disc} below) with conforming and
divergence-free pairs, i.e. under the assumptions $\SPveldisc
\subseteq \SPvel$ and $\Div \SPveldisc = \SPpresdisc$. Such pairs have attracted a growing interest in recent years; see 
\cite{Guzman.Neilan:14,Guzman.Neilan:18,Scott.Vogelius:85,Zhang:07}
and the references therein. Owing to~\eqref{intro:div-free-est}, this class of discretizations seems particularly attractive, because it fully exploits, up to a constant, the approximation properties of the space $\SPveldisc$ in the $H^1$-norm. This prevents, in particular, from the following issues.  

For standard discretizations with general conforming pairs (see~\eqref{Stokes-disc} and~\eqref{conforing-disc} below) one typically has
\begin{equation}
\label{intro:conforming-est}
\NormLeb{\Grad(u-u_h)}{}
\leq c \left( 
\inf_{w_h \in \SPveldisc} \NormLeb{\Grad(u-w_h)}{}
+\dfrac{1}{\Visc}
\inf_{q_h \in \SPpresdisc} \NormLeb{p-q_h}{}
\right).
\end{equation}
Thus, if $\Div \SPveldisc \neq \SPpresdisc$, the right-hand side suggests that the velocity $H^1$-error may be not robust with respect to the pressure. This is indeed the case and such effect is known in the literature as poor mass conservation. It becomes extreme for purely irrotational loads or for small values of the viscosity; see, for instance,~\cite{Linke:14}. Poor mass conservation discourages, in particular, from the use of unbalanced pairs, i.e. pairs $\SPveldisc/\SPpresdisc$ so that the approximation power of $\SPveldisc$ in the $H^1$-norm is higher than the one of $\SPpresdisc$ in the $L^2$-norm; cf. Remark~\ref{R:unbalanced-pairs}. 

Recall also that, in the nonconforming case $\SPveldisc \nsubseteq \SobH{\Domain}^\Dim$, estimates in the form
\begin{equation}
\label{intro:nonconforming-est}
\Normh{u-u_h}
\leq c\left( 
\inf_{w_h \in \SPveldisc} \Normh{u-w_h}
+ \dfrac{1}{\Visc}
\inf_{q_h \in \SPpresdisc} \NormLeb{p-q_h}{}
+
\Normtr{(u,p)}_h
\right) 
\end{equation}
are often derived. Here $\Normh{\cdot}$ is an extension of the
$H^1$-norm to $\SobH{\Domain}^\Dim + \SPveldisc$ and the semi-norm
$\Normtr{\cdot}_h$ is defined on (a subspace of) $\SobH{\Domain}^\Dim
\times \Leb{\Domain}$. Since the lack of smoothness in $\SPveldisc$ is
commonly compensated by additional regularity of the load beyond
$\SobD{\Domain}^\Dim$, the semi-norm $\Normtr{\cdot}_h$ cannot be
extended to $\SobH{\Domain}^\Dim \times \LebH{\Domain}$ and
potentially dominates the right-hand side of
\eqref{intro:nonconforming-est} for rough solutions. Therefore, an
estimate like~\eqref{intro:conforming-est} cannot be expected to hold,
cf. Remark~\ref{R:smoothing-E}.  

Several techniques are available in the literature to deal with the
above mentioned difficulties. The discretization of
\cite[section~6]{Badia.Codina.Gudi.Guzman:14} and the general
framework in~\cite{Veeser.Zanotti:18} indicate how to avoid the issue
with $\Normtr{\cdot}_h$ for nonconforming pairs. The over-penalized
augmented Lagrangian formulation of~\cite{Boffi.Lovadina:97} and the
grad-div stabilization~\cite{Olshanskii.Reusken:04} may serve to
mitigate the impact of poor mass conservation. More recently, Linke et
al.~\cite{Lederer.Linke.Merdon.Schoberl:17,Linke:14,Linke.Matthies.Tobiska:16}
proposed a class of discretizations, which differ from standard ones
only in the treatment of the load and enjoy the following pressure
robust upper bound 
\begin{equation}
\label{intro-Linke-est}
\Normh{u-u_h}
\leq c \left(
\inf_{w_h \in \SPveldisc} \Normh{u-w_h} + \Normtr{(u, 0)}_h
\right) 
\end{equation}
for several conforming and nonconforming pairs.

In this paper, we show that the quasi-optimal and
pressure robust estimate~\eqref{intro:div-free-est} is not a
prerogative of conforming and divergence-free pairs, but can be
achieved also by (carefully designed) discretizations, based on
general inf-sup stable pairs. In this way, we combine the advantages
of the various techniques listed above. We also bound the pressure 
$L^2$-error only in terms of the best approximation errors to the
analytical velocity and to the analytical pressure. To our best
knowledge, similar error bounds were previously obtained only in
\cite{Verfuerth.Zanotti:18} in the rather specific case of the
lowest-order nonconforming Crouzeix-Raviart pair
\cite{Crouzeix.Raviart:73}. In particular, our results make unbalanced
pairs a valuable option, if one is more interested in the analytical
velocity rather than in the analytical pressure. 

Our approach is guided by few simple necessary conditions and builds on two main ingredients. First, we discretise the load with
the help of an operator which maps $\SPveldisc$ into $\SPvel$ and
discretely divergence-free into exactly divergence-free functions. The
importance of the latter property was first devised in
\cite{Linke:14}. For this purpose, we solve local Stokes problems with
Scott-Vogelius elements on a barycentric refinement of the mesh, see
\cite{Guzman.Neilan:18,Qin:1994,Zhang:05}. Second, we discretise the weak form of the Laplace operator in a way inspired by Discontinuous Galerkin (DG) methods, in order to enforce the necessary consistency. The resulting discretization can be interpreted as a new augmented Lagrangian formulation, cf. Remark~\ref{R:connection-augmented-lagrangian}.

The rest of the paper is organized as follows. In
section~\ref{S:abstract-framework} we set up the abstract framework. In section~\ref{S:paradigmatic-discretization} we illustrate our
construction by means of a model example. Various generalizations are then discussed in section~\ref{S:generalizations}. Finally, in section~\ref{S:numerics} we complement our theoretical findings through some numerical experiments.

\section{Abstract framework}
\label{S:abstract-framework}

This section introduces an abstract discretization of
\eqref{Stokes-strong} and the properties in which we are
interested. Two basic results are also proved. We use standard
notations for Lebesgue and Sobolev spaces. 

\subsection{Quasi-optimal discretizations}
\label{SS:quasi-optimality}

Let $\Domain \subseteq \R^\Dim$, $\Dim \in \{2,3\}$, be an open and
bounded polytopic domain with Lipschitz-continuous boundary. The weak
formulation of the stationary Stokes equations in $\Domain$, with
viscosity $\Visc > 0$ and load $f \in \SobD{}^\Dim$, looks for $u \in
\SPvel$ and $p \in \SPpres$ 
such that 
\begin{equation}
\label{Stokes-weak}
\begin{alignedat}{2}
&\forall v \in \SPvel&
\qquad
\Visc \int_\Domain \Grad u \colon \Grad v 
- \int_\Domain p \Div v 
&= \left\langle  f , v \right\rangle  \\
&\forall q \in \SPpres&
\qquad
\int_\Domain q \Div u &= 0 .
\end{alignedat}
\end{equation}
Here $\colon$ denotes the euclidean scalar product of $\Dim \times
\Dim$ tensors and $\left\langle \cdot, \cdot \right\rangle $ is the
dual pairing of $\SobD{}^\Dim$ and $\SPvel$. Due to the boundary
condition on the analytical velocity $u$, the analytical pressure $p$
belongs to $\SPpres := \{ q \in \Leb{} \mid \int_\Domain q = 0
\}$. Problem~\eqref{Stokes-weak} is uniquely solvable, according to
\cite[Theorem~8.2.1]{Boffi:Brezzi:Fortin.13}. 

\begin{remark}[Alternative formulation]
\label{R:alternative-formulation}
Most of our subsequent results remain unchanged in case the gradient
is replaced by the symmetric gradient in the first equation of
\eqref{Stokes-weak} and the homogeneous Neumann condition is imposed
on (a portion of) $\partial \Domain$. The only remarkable difference
is that a piecewise Korn's inequality may fail to hold for some of the
nonconforming pairs mentioned in section~\ref{SS:nonconforming-pairs},
see~\cite{Arnold:93,Brenner:04}. This problem, however, can be
overcome e.g. by an additional jump penalization in the spirit of
\cite[Section~3.3]{Veeser.Zanotti:18b}. 
\end{remark}

We consider discretizations that mimic the variational structure
of problem~\eqref{Stokes-weak}. More precisely, we approximate $u$ and
$p$ in finite-dimensional linear spaces $\SPveldisc$ and
$\SPpresdisc$. We require $\SPpresdisc \subseteq \SPpres$ and measure
the pressure error in the $L^2$-norm $\NormLeb{\cdot}{}$. Instead, we
allow for nonconforming discrete velocity spaces $\SPveldisc
\nsubseteq \SPvel$. In order to measure the velocity error, we assume
that an extension $\Normh{\cdot}$ of the $H^1$-norm $\NormLeb{\Grad
  \cdot}{}$ to $\SPvel + \SPveldisc$ is at our disposal. We replace the 
bilinear forms in~\eqref{Stokes-weak} with discrete surrogates
$\Formveldisc: \SPveldisc \times \SPveldisc \to \R$ and $\Formpresdisc
: \SPveldisc \times \SPpresdisc \to \R$. Moreover, we let $\Smt:
\SPveldisc \to \SPvel$ be a linear operator. Hence, we look for a discrete velocity $u_h \in \SPveldisc$ and a discrete
pressure $p_h \in \SPpresdisc$ such that 
\begin{equation}
\label{Stokes-disc}
\begin{alignedat}{2}
&\forall v_h \in \SPveldisc
&\qquad
\Visc \,\Formveldisc(u_h, v_h)
+ \Formpresdisc(v_h, p_h) 
&= \left\langle  f , \Smt v_h \right\rangle  \\
&\forall q_h \in \SPpresdisc
&\qquad
\Formpresdisc(u_h, q_h) &= 0 .
\end{alignedat}
\end{equation}

To ensure that this problem is uniquely solvable, we assume hereafter that
$\Formveldisc$ is coercive on $\SPveldisc$ and that the pair
$\SPveldisc / \SPpresdisc$ is inf-sup stable, i.e.  
\begin{equation}
\label{inf-sup-disc}
\forall q_h \in \SPpresdisc 
\qquad
\beta \NormLeb{q_h}{} \leq
\sup_{v_h \in \SPveldisc} \dfrac{\Formpresdisc(v_h, q_h)}{\Normh{v_h}}
\end{equation}
for some constant $\beta > 0$, see
\cite[Corollary~4.2.1]{Boffi:Brezzi:Fortin.13}. Note, in particular,
that the duality $\left\langle  f , \Smt v_h \right\rangle$ is
well-defined for all $f \in \SobD{}^\Dim$ and $v_h \in \SPveldisc$,
also in the nonconforming case. 

We shall pay special attention to the following property, which
guarantees that $(u_h, p_h)$ is a near-best approximation of $(u,p)$
in $\SPveldisc \times \SPpresdisc$. 

\begin{definition}[Quasi-optimality]
\label{D:quasi-optimal}
Denote by $(u,p)$ and $(u_h,p_h)$ the solutions of~\eqref{Stokes-weak}
and~\eqref{Stokes-disc}, respectively, with load $f$ and viscosity
$\Visc$. We say that~\eqref{Stokes-disc} is a quasi-optimal
discretization of~\eqref{Stokes-weak} when there is a constant $C \geq
1$ such that 
\begin{equation}
\label{quasi-optimality}
\Visc \Normh{u-u_h}
+
\NormLeb{p-p_h}{}
\leq C \left( 
\Visc  \inf_{w_h \in \SPveldisc} \Normh{u-w_h} +
\inf_{q_h \in \SPpresdisc} \NormLeb{p-q_h}{}
\right) 
\end{equation}
for all $f \in \SobD{}^\Dim$ and $\Visc > 0$. We denote by $\Cqo$ the smallest such constant.
\end{definition}

According to~\cite[Theorem~5.2.5]{Boffi:Brezzi:Fortin.13}, the
discretization~\eqref{Stokes-disc} is quasi-optimal if 
\begin{equation}
\begin{gathered}
\label{conforing-disc}
\SPveldisc \subseteq \SPvel \qquad \Smt = \Id_{\SPveldisc}\\
\Formveldisc(w_h, v_h) = \int_\Domain \Grad w_h \colon \Grad v_h
\qquad
\Formpresdisc(v_h, q_h) = -\int_\Domain q_h \Div v_h
\end{gathered}
\end{equation}
i.e. if $\SPveldisc/\SPpresdisc $ is a conforming pair and
$\Formveldisc$, $\Formpresdisc$ and $\Smt$ are simple restrictions of
their conforming counterparts in~\eqref{Stokes-weak}. In
sections~\ref{S:paradigmatic-discretization} and
\ref{S:generalizations} we show that quasi-optimality can be achieved
also with nonconforming pairs and/or for different choices of
$\Formveldisc$ and $\Smt$.  

\begin{remark}[Smoothing by $\Smt$]
\label{R:smoothing-E}
Since $\SPveldisc$ is finite-dimensional, the operator $\Smt$ is
bounded and the solution of~\eqref{Stokes-disc} depends continuously
on the $H^{-1}$-norm of $f$. This property, in turn, prevents the
issue pointed out in the introduction concerning the semi-norm
$\Normtr{\cdot}_h$ in~\eqref{intro:nonconforming-est}. Of course, such
observation is of practical interest only if the norm of $\Smt$ is of
moderate size, so that it does not affect too much the stability
constant of~\eqref{Stokes-disc}. We call $\Smt$ "smoothing" operator,
because it increases the smoothness of the elements of $\SPveldisc$
whenever $\SPveldisc \nsubseteq \SPvel$. For conforming pairs, one can
let $\Smt$ be the identity as in~\eqref{conforing-disc}. This choice
is compatible with quasi-optimality but, possibly, it is not pressure
robust; compare with section~\ref{SS:quasi-optimality-press-robustness} below.

\end{remark}

\begin{remark}[Computational feasibility]
\label{R:computational-feasibility}
It is highly desirable that there are bases $\{ \varphi_1, \dots,
\varphi_N \}$ and $\{ \psi_1, \dots, \psi_M \}$ of $\SPveldisc$ and
$\SPpresdisc$, respectively, such that the scalars 
\begin{equation*}
\label{computational-feasibility}
\Formveldisc(\varphi_i, \varphi_j) 
\qquad
\Formpres(\varphi_i, \psi_k)
\qquad
\left\langle f, \Smt \varphi_i \right\rangle 
\end{equation*}
can be computed or approximated, up to a prescribed tolerance, with $O(1)$ operations, for all $i,j =
1,\dots, N$ and $k=1,\dots, M$. This "computational feasibility" is not necessary for quasi-optimality but guarantees that the solution of~\eqref{Stokes-disc} can be computed with optimal complexity. 
\end{remark}

\subsection{Quasi-optimal and pressure robust discretizations}
\label{SS:quasi-optimality-press-robustness}

The analytical velocity $u$ solving~\eqref{Stokes-weak} can be
equivalently characterized as the solution of an elliptic problem. In
fact, the second equation imposes that $u$ is divergence-free or, in
other words, that it is an element of the kernel 
\begin{equation*}
\label{kernel}
\SPker :=
\{ z \in \SPvel \mid \Div z = 0 \}.
\end{equation*} 
Then, testing the first equation with an arbitrary element of
$\SPker$, we obtain the reduced problem 
\begin{equation}
\label{Stokes-reduced}
\forall z \in \SPker \qquad
\Visc \int_\Domain \Grad u \colon \Grad z 
= \left\langle f ,z \right\rangle 
\end{equation}
which is uniquely solvable, according to the Lax-Milgram lemma and the 
Friedrichs inequality. 

The same structure can be observed at the discrete level. To see this,
we first introduce the discrete divergence $\Divdisc: \SPveldisc \to
\SPpresdisc$ by
\begin{equation}
\label{divergence-disc}
\forall q_h \in \SPpresdisc \qquad
\int_\Domain q_h \Divdisc v_h
= - \Formpresdisc(v_h, q_h)
\end{equation}
for all $v_h \in \SPveldisc$. The second equation of~\eqref{Stokes-disc} imposes that $u_h$ is discretely divergence-free, i.e. it is an element of the discrete kernel
\begin{equation*}
\label{kernel-discrete}
\SPkerdisc :=
\{ z_h \in \SPveldisc \mid \Divdisc z_h = 0 \}.
\end{equation*}
Then, testing the first equation with an arbitrary element of $\SPkerdisc$, we derive the discrete reduced problem
\begin{equation}
\label{Stokes-reduced-disc}
\forall z_h \in \SPkerdisc \qquad
\Visc\,  \Formveldisc(u_h, z_h) 
= \left\langle f, \Smt z_h \right\rangle 
\end{equation}
which is uniquely solvable, since $\Formveldisc$ is coercive on
$\SPveldisc$. In the vein of~\cite[Remark~2.1]{Brezzi:74}, it is worth
recalling that this is a (possibly) nonconforming discretization of
\eqref{Stokes-reduced}, because $\SPkerdisc$ may fail to be a subspace of $\SPker$, even if $\SPveldisc \subseteq \SPvel$. 

Similarly as in Definition~\ref{D:quasi-optimal}, we will be interested in
the question whether $u_h$ is a near-best approximation of $u$ in
$\SPkerdisc$. This actually amounts to ask whether $u_h$ is near-best
in $\SPveldisc$, because the inf-sup condition~\eqref{inf-sup-disc}
implies 
\begin{equation}
\label{best-errors-Vh-Zh}
\inf_{z_h \in \SPkerdisc} \Normh{u-z_h}
\leq \left( 1 + \beta^{-1} \right) 
\inf_{w_h \in \SPveldisc} \Normh{u-w_h} 
\end{equation} 
according to~\cite[Proposition~5.1.3]{Boffi:Brezzi:Fortin.13} and~\cite[Lemma~2.1]{Pyo.Nochetto:05}.

\begin{definition}[Quasi-optimality and pressure robustness]
\label{D:quasi-optimality-press-robust}
Denote by $u$ and $u_h$ the solutions of~\eqref{Stokes-reduced} and
\eqref{Stokes-reduced-disc}, respectively, with load $f$ and viscosity
$\Visc$. We say that~\eqref{Stokes-disc} is a quasi-optimal and
pressure robust discretization of~\eqref{Stokes-weak} when there is a
constant $C \geq 1$ such that 
\begin{equation}
\label{quasi-optimal-press-robust}
\Normh{u-u_h}
\leq C
\inf_{w_h \in \SPveldisc} \Normh{u-w_h}
\end{equation}
for all $f \in \SobD{}^\Dim$ and $\Visc > 0$. We denote by $\Cqopr$ the smallest such constant.
\end{definition}

Problem~\eqref{Stokes-reduced} reveals that the analytical velocity
$u$ is independent of the pressure $p$ and depends on the load $f$
only through its restriction to $\SPker$. This implies, for instance,
that $u$ is invariant with respect to irrotational perturbations of
$f$, see Linke~\cite{Linke:14}. The near-best estimate
\eqref{quasi-optimal-press-robust} guarantees that $u_h$ reproduces
such invariance property at the discrete level and justifies the
designation "pressure robust".  

The discretization~\eqref{Stokes-disc} is known to be quasi-optimal and pressure robust if 
\begin{equation}
\begin{gathered}
\label{conforing-div-free-disc}
\SPveldisc \subseteq \SPvel \qquad
\Div \SPveldisc = \SPpresdisc \qquad
\Smt = \Id_{\SPveldisc}\\
\Formveldisc(w_h, v_h) = \int_\Domain \Grad w_h \colon \Grad v_h
\qquad
\Formpresdisc(v_h, q_h) = -\int_\Domain q_h \Div v_h
\end{gathered}
\end{equation}
i.e. if $\SPveldisc/ \SPpresdisc$ is a conforming and divergence-free
pair and $\Formveldisc$, $\Formpresdisc$ and $\Smt$ are simple
restrictions of their continuous counterparts in
\eqref{Stokes-weak}. In fact, in this case, we have $\SPkerdisc
\subseteq \SPker$ and~\eqref{Stokes-reduced-disc} is a conforming
Galerkin discretization of~\eqref{Stokes-reduced}. Therefore,
C\'{e}a's lemma and~\eqref{best-errors-Vh-Zh} imply $\Cqopr \leq (1+
\beta^{-1})$. It is our purpose to show that quasi-optimality and
pressure robustness can be achieved also by other discretizations than~\eqref{conforing-div-free-disc}.  

\subsection{Necessary consistency conditions}
\label{SS:necessary-consistency}

The left- and the right-hand sides of \eqref{quasi-optimality} are
seminorms on $\SPker \times \SPpres$ and the kernel of the latter is
$(\SPker \cap \SPkerdisc) \times \SPpresdisc$, as a consequence of
\eqref{best-errors-Vh-Zh}. Quasi-optimality actually prescribes that such seminorms are equivalent, because the converse of
\eqref{quasi-optimality} immediately follows from the inclusion $(u_h, p_h) \in \SPkerdisc
\times \SPpresdisc$. Hence, a simple necessary condition is that the
kernels of the two seminorms coincide. In other words, whenever the solution $(u,p)$ of \eqref{Stokes-weak} is in $\SPkerdisc \times \SPpresdisc$, it must solve also \eqref{Stokes-disc}. This is an algebraic
consistency condition, which can be rephrased in terms of the forms
$\Formveldisc$ and $\Formpresdisc$ and of the operator $\Smt$, in the
spirit of~\cite[Definition~2.7]{Veeser.Zanotti:18}. 

\begin{lemma}[Consistency for quasi-optimality]
\label{L:quasi-optimal-nec} 
Assume that~\eqref{Stokes-disc} is a quasi-optimal discretization of
\eqref{Stokes-weak}. Then, necessarily we have
\begin{subequations}
\label{quasi-optimal-nec}
\begin{alignat}{2}
\label{quasi-optimal-nec-div}
&\forall v_h \in \SPveldisc, \, p \in \SPpresdisc &\qquad
&\int_\Domain p ( \Divdisc v_h - \Div \Smt v_h ) = 0
\intertext{and}\label{quasi-optimal-nec-lapl}
&\forall u \in \SPker \cap \SPkerdisc, \, v_h \in \SPveldisc &\qquad&
\Formveldisc(u, v_h) = \int_\Domain \Grad u \colon \Grad \Smt v_h.
\end{alignat}
\end{subequations}
\end{lemma}

\begin{proof}
Denote by $(u,p)$ the solution of~\eqref{Stokes-weak} and assume first $u=0$ and $p \in \SPpresdisc$. Quasi-optimality
implies that the solution $(u_h, p_h)$ of~\eqref{Stokes-disc}
satisfies $u_h = 0$ and $p_h = p$. Comparing the first equations of
\eqref{Stokes-weak} and~\eqref{Stokes-disc}, we derive the identity
$\Formpresdisc(v_h, p) = -\int_\Domain p \Div \Smt v_h $ for all $v_h
\in \SPveldisc$. Condition~\eqref{quasi-optimal-nec-div} then follows
from the definition of $\Divdisc$ in~\eqref{divergence-disc}. Next,
assume $u \in \SPker \cap \SPkerdisc$ and $p=0$. Since
quasi-optimality implies $u_h = u$ and $p_h = 0$, condition
\eqref{quasi-optimal-nec-lapl} can be derived comparing the first
equations of~\eqref{Stokes-weak} and~\eqref{Stokes-disc} as before.    
\end{proof}

The conforming discretization~\eqref{conforing-disc} is a simple
option to fulfill~\eqref{quasi-optimal-nec}, but not the only
possible.
Examples with nonconforming discrete velocity space can be found
in~\cite[Section~6]{Badia.Codina.Gudi.Guzman:14} and
\cite{Verfuerth.Zanotti:18}. Standard nonconforming discretizations,
like the one of Crouzeix and Raviart~\cite{Crouzeix.Raviart:73}, do
not fulfill~\eqref{quasi-optimal-nec}, because they do not
employ a smoothing operator. It is also worth noticing
that~\eqref{quasi-optimal-nec} involves the interplay of
$\Formveldisc$ and $\Formpresdisc$ with $\Smt$. This indicates that the discretization of the differential operator in~\eqref{Stokes-strong} and the one of the corresponding load should not be regarded as independent tasks.

Proceeding similarly as in Lemma~\ref{L:quasi-optimal-nec}, we derive
necessary conditions for quasi-optimality and pressure robustness. 

\begin{lemma}[Consistency for quasi-optimality and pressure robustness]
\label{L:quasi-optimal-press-robust-nec}
Assume that~\eqref{Stokes-disc} is a quasi-optimal and pressure robust
discretization of~\eqref{Stokes-weak}. Then, necessarily we have 
\begin{subequations}
	\label{quasi-optimal-press-robust-nec}
	\begin{equation}
	\label{quasi-optimal-press-robust-nec-div}
	\Smt (\SPkerdisc) \subseteq \SPker
	\end{equation}
	and
	\begin{equation}
	\label{quasi-optimal-press-robust-nec-lapl}
	\forall u \in \SPker \cap \SPkerdisc, \, z_h \in \SPkerdisc \qquad
	\Formveldisc(u, z_h) = \int_\Domain \Grad u \colon \Grad \Smt z_h.
	\end{equation}
\end{subequations}
\end{lemma}

\begin{proof}
Let $z_h \in \SPkerdisc$ be such that $\Div \Smt z_h \neq 0$.
Assuming that $(u,p) = (0, \Div \Smt z_h)$ solves
\eqref{Stokes-weak}, we infer $\left\langle f, \Smt z_h \right\rangle = -\NormLeb{\Div \Smt z_h}{}^2 \neq 0$. Inserting this information
in~\eqref{Stokes-reduced-disc}, we obtain $u_h \neq
0$. Therefore, we have $\Normh{u-u_h} > \inf_{v_h \in \SPveldisc}
\Normh{u-v_h} = 0$, which contradicts quasi-optimality and
pressure robustness. This proves
\eqref{quasi-optimal-press-robust-nec-div}. Assertion
\eqref{quasi-optimal-press-robust-nec-lapl} may be checked similarly to~\eqref{quasi-optimal-nec-lapl} in Lemma~\ref{L:quasi-optimal-nec}. 
\end{proof}

Condition~\eqref{quasi-optimal-press-robust-nec-lapl} is clearly
necessary for~\eqref{quasi-optimal-nec-lapl}, while
\eqref{quasi-optimal-press-robust-nec-div} is neither necessary nor
sufficient for~\eqref{quasi-optimal-nec-div}. We mention also that
\eqref{quasi-optimal-press-robust-nec-div} differs from the condition
exploited in~\cite{Linke.Matthies.Tobiska:16} to achieve pressure
robustness, in that here $\Smt$ is required to map into $\SPvel$ and
not only into $\Hdiv{}$, cf. Remark~\ref{R:smoothing-E}. 

\begin{remark}[Failure of $\Smt=\Id_{\SPveldisc}$]
\label{R:failure-identity}
If $\SPveldisc/\SPpresdisc$ is a conforming and divergence-free pair,
the abstract discretization~\eqref{Stokes-disc} with
\eqref{conforing-div-free-disc} verifies the first necessary
condition in Lemma~\ref{L:quasi-optimal-press-robust-nec}. If,
instead, the pair is conforming but not divergence-free, we have
$\SPkerdisc \nsubseteq \SPker$. In this case, the operator $\Smt$
cannot coincide with the identity on $\SPkerdisc$. 
\end{remark}

In the next sections, we design some new discretizations proceeding as
follows. Given an inf-sup stable pair $\SPveldisc/\SPpresdisc$, together with the corresponding bilinear form $\Formpresdisc$, we construct $\Formveldisc$ and $\Smt$
so that the necessary conditions in Lemmas~\ref{L:quasi-optimal-nec}
and~\ref{L:quasi-optimal-press-robust-nec} hold true. Then, we use
standard techniques from the analysis of saddle point problems to
verify~\eqref{quasi-optimality} and~\eqref{quasi-optimal-press-robust}
and to bound the constants $\Cqo$ and $\Cqopr$. Alternatively, one
could exploit~\cite[Theorem~4.14]{Veeser.Zanotti:18}, which guarantees
that~\eqref{quasi-optimal-press-robust-nec} is a sufficient condition for quasi-optimality and pressure robustness. Such result provides also a
formula for $\Cqopr$. Analogously, generalizing the framework of
\cite{Veeser.Zanotti:18}, one could show also that
\eqref{quasi-optimal-nec} is a sufficient condition for quasi-optimality and
derive a formula for $\Cqo$. We prefer to proceed as indicated, to
make sure this paper can be read independently of
\cite{Veeser.Zanotti:18}.

\section{A paradigmatic discretization}
\label{S:paradigmatic-discretization}

Assume that we are given an inf-sup stable pair
$\SPveldisc/\SPpresdisc$, together with the corresponding bilinear
form $\Formpresdisc$. A possible strategy to fulfill the necessary
conditions~\eqref{quasi-optimal-nec-div} and
\eqref{quasi-optimal-press-robust-nec-div} is to employ a
"divergence-preserving" smoothing operator, i.e. 
\begin{equation}
\label{conservation-divergence}
\forall v_h \in \SPveldisc \qquad
\Div \Smt v_h = \Divdisc v_h.
\end{equation}
Once such operator is given, conditions~\eqref{quasi-optimal-nec-lapl}
and~\eqref{quasi-optimal-press-robust-nec-lapl} prescribe the
restriction of $\Formveldisc$ on $(\SPker \cap \SPkerdisc) \times
\SPveldisc$. Then, inspired by~\cite{Arnold:82} and~\cite{Veeser.Zanotti:18b}, we extend
the resulting form to $\SPveldisc\times \SPveldisc$, in a way that
additionally ensures symmetry and coercivity.  
%
In order to keep the exposition as clear as possible, we first exemplify this idea in a model setting. We postpone various generalizations to the
next section. 

\subsection{The unbalanced $\Poly{\Degree} / \Poly{\Degree-2}$ pair} 
\label{SS:Pl-Pl-2-pair}

We  consider hereafter pairs of finite element spaces on a
face-to-face simplicial mesh $\Mesh$ of $\Domain$ in the sense
of~\cite[Definition 1.36]{DiPietro.Ern:12}. We write $c$ for a
nondecreasing and nonnegative function of the shape parameter of
$\Mesh$, which possibly depends also on different quantities (like, e.g.,
the space dimension), but neither on other properties of $\Mesh$ nor
on the viscosity $\Visc$. Such constant may change at different
occurrences. We occasionally abbreviate $a \leq c b $ as $a \Cleq b$ and
$c^{-1} b \leq a \leq c b$ as $a \eqsim b$.  

For all integers $\Degree \geq 0$, we denote by $\Poly{\Degree}(S)$
the space of polynomials with total degree $\leq \Degree$ on a simplex
$S \subseteq \R^\Dim$. The space of $H^k$-conforming element-wise
polynomials on $\Mesh$ then reads 
\begin{equation}
\label{elementwise-polynomials}
\Polypiec{\Degree}{k}
:=
\{ v \in H^k(\Domain) \mid \forall K \in \Mesh \;\; v_{|K} \in \Poly{\Degree}(K) \}
\end{equation}  
with $k \in \{0,1\}$ and the convention $H^0(\Domain) :=
\Leb{}$. Motivated by the homogeneous boundary condition in
\eqref{Stokes-strong}, we consider the subspaces 
\begin{equation}
\label{elementwise-polynomials-sub}
\Polybnd{\Degree}{1} :=
\Polypiec{\Degree}{1} \cap \SobH{}
\qquad \text{and} \qquad
\Polyavg{\Degree}{k} :=
\Polypiec{\Degree}{k} \cap \LebH{}.
\end{equation}

To exemplify our construction, we assume $\Dim=2$ for the remaining
part of this section. We consider the conforming $\Poly{\Degree} /
\Poly{\Degree-2}$ pair, which is given by 
\begin{equation}
\label{Pl-Pl-2-pair}
\SPveldisc = (\Polybnd{\Degree}{1})^2 
\qquad \text{and} \qquad
\SPpresdisc = \Polyavg{\Degree-2}{0},
\qquad
\Formpresdisc(v_h, q_h) = -\int_\Domain q_h \Div v_h
\end{equation}
with $\Degree \geq 2$. The inf-sup condition~\eqref{inf-sup-disc}
holds with $\beta^{-1} \leq c$, see
\cite[Remark~8.6.2]{Boffi:Brezzi:Fortin.13}. 

\begin{remark}[Unbalanced pairs]
\label{R:unbalanced-pairs}
The $\Poly{\Degree}/\Poly{\Degree-2}$ pair is unbalanced, in the sense that the approximation power $\ell-1$ of the discrete pressure space in the $L^2$-norm is strictly less than the approximation power $\Degree$ of the discrete velocity space in the
$H^1$-norm. Other examples can be obtained enriching the velocity space of any
inf-sup stable pair. The use of conforming unbalanced pairs, in
combination with the standard discretization~\eqref{conforing-disc},
is discouraged by the error estimate~\eqref{intro:conforming-est} and
Remark~\ref{R:failure-identity}; see also
\cite[Remark~8.6.2]{Boffi:Brezzi:Fortin.13}. Still, quasi-optimal and
pressure robust discretizations based on such pairs would be a
valuable option, if one is more interested in the analytical velocity
rather than in the analytical pressure. 
\end{remark} 

The discrete divergence $\Divdisc$ in the $\Poly{\Degree}/\Poly{\Degree-2}$ pair
coincides with the $L^2$-orthogonal projection of the analytical divergence onto $\Polyavg{\Degree-2}{0}$. Since~\eqref{divergence-disc} actually holds for all
discrete pressures in $\Polypiec{\Degree-2}{0}$, we can compute $\Divdisc$ element-wise as
follows 
\begin{equation}
\label{divergence-disc-Pl-Pl-2}
\Divdisc v_h
=
\Ritz^K_{\Degree-2} \Div v_h 
\quad \text{in} \,\, K  
\end{equation}
for all $v_h \in (\Polybnd{\Degree}{1})^2$ and $K \in \Mesh$, where
$\Ritz^K_{\Degree-2}$ is the $L^2$-orthogonal projection onto
$\Poly{\Degree-2}(K)$. Therefore, denoting by $\SPkerdisc^\Pl$ the
discrete kernel, we conclude $\SPkerdisc^{\Pl} \nsubseteq \SPker$.\footnote{The superscript $"\Pl"$ stands for "unbalanced". Along this section, we use it to label spaces, forms and operators related to the $\Poly{\Degree}/\Poly{\Degree-2}$ pair.} This confirms that the $\Poly{\Degree}/\Poly{\Degree-2}$ pair is conforming but not divergence-free.  

The abstract discretization~\eqref{Stokes-disc} with~\eqref{conforing-disc}, based on the
$\Poly{\Degree}/\Poly{\Degree-2}$ pair, states $u_h \in
(\Polybnd{\Degree}{1})^2$ and $p_h \in \Polyavg{\Degree-2}{0}$ such
that 
\begin{equation}
\label{Stokes-Pl-Pl-2-standard}
\begin{alignedat}{2}
&\forall v_h \in (\Polybnd{\Degree}{1})^2
&\qquad
\Visc \,\int_\Domain \Grad u_h \colon \Grad v_h
-\int_\Domain p_h \Div v_h 
&= \left\langle  f , v_h \right\rangle  \\
&\forall q_h \in \Polyavg{\Degree-2}{0}
&\qquad
\int_\Domain q_h \Div u_h &= 0 .
\end{alignedat}
\end{equation}


\subsection{Local inversion of the divergence}
\label{SS:local-inversion-divergence}

Proceeding as in~\cite{Verfuerth.Zanotti:18}, we enforce
\eqref{conservation-divergence} with the help of local right inverses
of the divergence. Such operators can be defined through discrete
Stokes-like problems on the barycentric refinement of each element. To
see this, fix $K \in \Mesh$ and let $\Submesh_K$ denote the triangulation of $K$ obtained connecting each vertex with the barycenter; cf.
Figure~\ref{F:barycentric-refinement}. For $\Degree \in \N$, we define
the local spaces 
\begin{equation*}
\label{local-spaces}
\Polybnd{\Degree}{1}(\Submesh_K)
\qquad \text{and} \qquad
\Polyavg{\Degree-1}{0}(\Submesh_K)
\end{equation*}
on $\Submesh_K$ similarly to the global spaces
$\Polybnd{\Degree}{1}$ and $\Polyavg{\Degree-1}{0}$ in
\eqref{elementwise-polynomials-sub}. In particular, all $v_k \in
\Polybnd{\Degree}{1}(\Submesh_K)$ vanish on $\partial K$ and all $q_K
\in \Polyavg{\Degree-1}{0}(\Submesh_K)$ are such that $\int_K q_K =
0$. The pair $\Polybnd{\Degree}{1}(\Submesh_K)^2
/\Polyavg{\Degree-1}{0} (\Submesh_K)$ is conforming and
divergence-free in $K$.   

\begin{figure}[ht]
	\centering
	\begin{tikzpicture}
	\coordinate (z1) at (0,0);
	\coordinate (z2) at (2,0);
	\coordinate (z3) at (1,1.5);
	\coordinate (c1) at (1, 0.5);
	\path (z1) edge (z2);
	\path (z2) edge (z3);
	\path (z3) edge (z1);
	\coordinate (z4) at (4,0);
	\coordinate (z5) at (6,0);
	\coordinate (z6) at (5,1.5);
	\coordinate (c2) at (5, 0.5);
	\path (z4) edge (z5);
	\path (z5) edge (z6);
	\path (z6) edge (z4);
	\path[dashed] (z4) edge (c2);
	\path[dashed] (z5) edge (c2);
	\path[dashed] (z6) edge (c2);
	\end{tikzpicture}
	\caption{Generic element $K \in \Mesh$ (left) and barycentric refinement $\Submesh_K$ (right).}
	\label{F:barycentric-refinement}
\end{figure}

According to~\cite[Theorem~3.1]{Guzman.Neilan:18}, we have the local inf-sup stability
\begin{equation}
\label{local-inf-sup}
\forall q_K \in \Polyavg{\Degree-1}{0}(\Submesh_K) 
\qquad
\NormLeb{q_K}{K} \leq c
\sup_{v_K \in \Polybnd{\Degree}{1}(\Submesh_K)^2} \dfrac{\int_K q_K
  \Div v_K}{\NormLeb{\Grad v_K}{K}}. 
\end{equation}
This entails that we can define a linear operator
$\Rightinv_\Degree^{K}: \Leb{\Domain} \to \SobH{}^2$ as follows. Given
$q \in \Leb{}$, let $u_K = u_K(q) \in
\Polybnd{\Degree}{1}(\Submesh_K)^2$ and $p_K = p_K(q) \in
\Polyavg{\Degree-1}{0}(\Submesh_K)$ solve 
\begin{equation}
\label{local-problem-div}
\begin{alignedat}{2}
&\forall v_K \in \Polybnd{\Degree}{1}(\Submesh_K)^2
&\qquad
\int_K \Grad u_K \colon \Grad v_K 
- \int_K p_K \Div v_K &= 0\\
&\forall q_K \in \Polyavg{\Degree-1}{0}(\Submesh_K)
&\quad \;
\int_K q_K \Div u_K &= \int_K q_K q.
\end{alignedat}
\end{equation} 
Hence, we set
\begin{equation*}
\label{local-right-inverse}
\Rightinv_\Degree^{K} q := u_K \quad \text{in} \;\; K
\qquad \text{and} \qquad
\Rightinv_\Degree^{K} q := 0 \quad \text{in} \;\; \Domain \setminus K.
\end{equation*}

\begin{proposition}[Local right inverses]
\label{P:local-right-inverse}
Let $K \in \Mesh$ be a mesh element and $\Degree \in \N$. The operator $\Rightinv_\Degree^{K}$ is well-defined and, for all $q \in \Leb{}$, we have
	\begin{subequations}
		\label{local-right-inverse-prop}
		\begin{equation}
		\label{local-right-inverse-prop-stab}
		\NormLeb{\Grad \Rightinv_\Degree^{K} q}{}
		\leq c \NormLeb{q}{K}
		\end{equation}
		and
		\begin{equation}
		\label{local-right-inverse-prop-div}
		q_{|K} \in \Polyavg{\Degree-1}{0}(\Submesh_K)
		\quad \Longrightarrow \quad
		\Div \Rightinv_\Degree^{K} q = q \quad \text{in} \;\; K
		\end{equation}
	\end{subequations}
\end{proposition}

\begin{proof}
	The operator $\Rightinv_\Degree^{K}$ is well-defined and
        satisfies~\eqref{local-right-inverse-prop-stab} in view of the
        local inf-sup~\eqref{local-inf-sup}
        and~\cite[Corollary~4.2.1]{Boffi:Brezzi:Fortin.13}. The
        property 
        in~\eqref{local-right-inverse-prop-div} directly follows from
        the second equation of problem~\eqref{local-problem-div},
        because $\Div u_K \in \Polyavg{\Degree-1}{0}(\Submesh_K) $. 
\end{proof}

\begin{remark}[Computation of the local right inverses]
\label{R:computation-rightinv}
In what follows, we shall need to compute $\Rightinv_\Degree^{K} q$
for all $K \in \Mesh$ and various $q \in \Polypiec{\Degree-1}{0}$. To
this end, a 
possible strategy is to 
precompute the solution of~\eqref{local-problem-div} on a reference
triangle $K_\Refel$, for all possible loads $q_\Refel$ in a basis of
$\Poly{\Degree-1}(K_\Refel)$. The computational complexity of this
task only depends on $\Degree$. Then, the solution of
\eqref{local-problem-div} in $K$ can be obtained in terms of the
corresponding solution in $K_\Refel$, by means of the contravariant 
Piola transformation; see
\cite[Section~2.1.3]{Boffi:Brezzi:Fortin.13}.
\end{remark}   

We have considered here the two-dimensional case only to be consistent with the simplification introduced in section~\ref{SS:Pl-Pl-2-pair}. The same construction is actually possible
in any space dimension $\Dim \geq 2$.

\subsection{A new augmented Lagrangian formulation}
\label{SS:DG-like-formulation}

We now propose a new discretization of the Stokes equations, based on
the $\Poly{\Degree}/\Poly{\Degree-2}$ pair. The first ingredient of our construction is a
linear operator $\Smt^\Pl: (\Polybnd{\Degree}{1})^2 \to \SobH{}^2$
fulfilling~\eqref{conservation-divergence}. In view of
$\SPkerdisc^{\Pl} \nsubseteq \SPker$ and
Remark~\ref{R:failure-identity}, the identity on
$(\Polybnd{\Degree}{1})^2$ cannot accommodate this
property. Therefore, we introduce a "divergence correction"
$\Rightinv_h^\Pl: (\Polybnd{\Degree}{1})^2 \to \SobH{}^2$ 
\begin{equation*}
\label{Pl-Pl-2-div-correction}
\Rightinv_h^\Pl v_h :=
\sum_{K \in \Mesh} \Rightinv_\Degree^K ( \Divdisc v_h - \Div v_h).
\end{equation*}  

\begin{proposition}[Divergence-preserving smoothing operator]
\label{P:Pl-Pl-2-smoother}
The linear operator $\Smt^\Pl: (\Polybnd{\Degree}{1})^2 \to \SobH{}^2$ given by
\begin{equation}
\label{Pl-Pl-2-smoother}
\Smt^\Pl v_h := 
v_h + \Rightinv_h^\Pl v_h
\end{equation}
fulfills~\eqref{conservation-divergence} and is such that, for all $v_h \in (\Polybnd{\Degree}{1})^2$,
\begin{equation}
\label{Pl-Pl-2-smoother-stability}
\NormLeb{\Grad(v_h - \Smt^\Pl v_h)}{}
\eqsim
\NormLeb{\Divdisc v_h - \Div v_h}{}.
\end{equation}
\end{proposition}

\begin{proof}
For all $v_h \in (\Polybnd{\Degree}{1})^2$ and $K \in \Mesh$, it holds
\begin{equation*}
\Div \Smt^\Pl v_h
=
\Div v_h + \Div \Rightinv_\Degree^K (\Divdisc v_h - \Div v_h)
\quad \text{in} \,\, K.
\end{equation*}
In view of~\eqref{divergence-disc-Pl-Pl-2}, we have $\int_K (\Divdisc v_h - \Div v_h) = 0$. Since the inclusion $v_h \in (\Polybnd{\Degree}{1})^2$ implies also $(\Divdisc v_h -
\Div v_h)_{|K} \in \Poly{\Degree-1}(K)$, Proposition~\ref{P:local-right-inverse} and the identity above ensure
that $\Smt^\Pl$ fulfills~\eqref{conservation-divergence}. This, in turn, easily implies the lower bound "$\gtrsim$" in 
\eqref{Pl-Pl-2-smoother-stability}. The corresponding upper bound "$\Cleq$" is a consequence of the identity $\NormLeb{\Grad (v_h - \Smt^\Pl v_h)}{K} = \NormLeb{\Grad \Rightinv^K_\Degree v_h}{K}$, $K \in \Mesh$, combined with~\eqref{local-right-inverse-prop-stab}. 
\end{proof}

The second ingredient of our construction is a suitable bilinear form
$\Formveldisc$. Accounting for the definition of $\Smt^\Pl$ in~\eqref{Pl-Pl-2-smoother}, the necessary conditions
\eqref{quasi-optimal-nec-lapl} and
\eqref{quasi-optimal-press-robust-nec-lapl} prescribe 
\begin{equation}
\label{Pl-Pl-2-bilinear-form-restriction}
\Formveldisc(u, v_h) 
=
\int_\Domain \Grad u \colon \Grad v_h + 
\int_\Domain \Grad u \colon \Grad \Rightinv_h^\Pl v_h
\end{equation} 
for all $u \in \SPker \cap \SPkerdisc^{\Pl}$ and $v_h \in
(\Polybnd{\Degree}{1})^2$. A simple option would be to let the
right-hand side define $\Formveldisc$ on $(\Polybnd{\Degree}{1})^2
\times (\Polybnd{\Degree}{1})^2$. Still, it has to be noticed that the
second summand $\int_\Domain \Grad u \colon \Grad \Rightinv_h^\Pl v_h
= - \sum_{K\in \Mesh}\int_K \Lapl u \cdot \Rightinv_h^\Pl v_h $ cannot
be expected to vanish. Therefore, it obstructs the symmetry and,
possibly, also the nondegeneracy of $\Formveldisc$. To overcome this
problem, we observe that $\Rightinv_h^\Pl$ vanishes on $\SPker \cap
\SPkerdisc^{\Pl}$, according to
\eqref{Pl-Pl-2-smoother-stability}. This suggests to re-establish
symmetry and nondegeneracy mimicking the construction of the Symmetric
Interior Penalty (DG-SIP) discretization of second-order problems,
see~\cite{Arnold:82} or~\cite[section~4.2.1]{DiPietro.Ern:12}. Thus, we set $\Formveldisc =
\Formveldisc^\Pl$, where 
\begin{equation}
\label{Pl-Pl-2-bilinear-form}
\begin{split}
\Formveldisc^\Pl(w_h, v_h) 
:=
&\int_\Domain \Grad w_h \colon \Grad v_h + 
\int_\Domain \Grad w_h \colon \Grad \Rightinv_h^\Pl v_h +\\
&+\int_\Domain \Grad \Rightinv_h^\Pl w_h \colon \Grad v_h +
\eta \int_\Domain \Grad \Rightinv_h^\Pl w_h \colon \Grad \Rightinv_h^\Pl v_h
\end{split}
\end{equation}  
where $\eta > 0$ is a penalty parameter. Note that $\Formveldisc^\Pl$ fulfills~\eqref{Pl-Pl-2-bilinear-form-restriction}.

The abstract discretization~\eqref{Stokes-disc} with the $\Poly{\Degree}/\Poly{\Degree-2}$ pair, $\Formveldisc =
\Formveldisc^\Pl$ and $\Smt = \Smt^\Pl$ reads as follows: Find $u_h
\in (\Polybnd{\Degree}{1})^2$ and $p_h \in \Polyavg{\Degree-2}{0}$
such that  
\begin{equation}
\label{Stokes-Pl-Pl-2}
\begin{alignedat}{2}
&\forall v_h \in (\Polybnd{\Degree}{1})^2
&\qquad
\Visc \,\Formveldisc^\Pl(u_h, v_h)
-\int_\Domain p_h \Div v_h 
&= \langle  f , \Smt^\Pl v_h \rangle  \\
&\forall q_h \in \Polyavg{\Degree-2}{0}
&\qquad
\int_\Domain q_h \Div u_h &= 0 .
\end{alignedat}
\end{equation}

We begin our discussion on the new discretization by checking that a
solution $(u_h, p_h)$ exists and is unique. In view of the
above-mentioned inf-sup stability of the
$\Poly{\Degree}/\Poly{\Degree-2}$ pair, it suffices to prove that
$\Formveldisc^\Pl$ is coercive on $(\Polybnd{\Degree}{1})^2$. We proceed similarly as in~\cite[Lemma~4.1.2]{DiPietro.Ern:12}.

\begin{lemma}[Coercivity of $\Formveldisc^\Pl$]
\label{L:Pl-Pl-2-coercivity}
The bilinear form $\Formveldisc^\Pl$ is coercive on
$(\Polybnd{\Degree}{1})^2$ for all $\eta > 1$ and we have 
\begin{equation*}
\label{Pl-Pl-2-coercivity}
\Formveldisc^\Pl (v_h, v_h) \geq \left( 1 - \dfrac{1}{\eta}\right)
\NormLeb{\Grad v_h}{}^2 
\end{equation*}
for all $v_h \in (\Polybnd{\Degree}{1})^2$.
\end{lemma}

\begin{proof}
Let $v_h \in (\Polybnd{\Degree}{1})^2$. Setting $w_h = v_h$ in~\eqref{Pl-Pl-2-bilinear-form}, we obtain 
\begin{equation*}
\Formveldisc^\Pl (v_h, v_h)
=
\NormLeb{\Grad v_h}{}^2 + \eta
\NormLeb{\Grad \Rightinv_h^\Pl v_h}{}^2
+2 \int_\Domain \Grad v_h \colon \Grad \Rightinv_h^\Pl v_h.
\end{equation*}
The Cauchy-Schwartz and the weighted Young's inequality further
provide the upper bound $2\NormSemi{\int_\Domain \Grad v_h \colon \Grad \Rightinv_h^\Pl v_h} \leq \eta^{-1} \NormLeb{\Grad v_h}{}^2 + \eta \NormLeb{\Grad \Rightinv_h^\Pl v_h}{}^2$. Inserting this inequality into the previous identity concludes the proof.   
\end{proof}

Let us comment on the cost for assembling and solving the new discretization.

\begin{remark}[Feasibility of the new discretization]
\label{R:Pl-Pl-2-feasibility}
Assume that $\{\varphi_1, \dots, \varphi_N\}$ and $\{\psi_1, \dots,
\psi_M\}$ are nodal bases of $(\Polybnd{\Degree}{1})^2$ and
$\Polyavg{\Degree-2}{0}$, respectively. All functions $\varphi_i$ and
$\psi_k$, with $i=1,\dots, N$ and $k=1,\dots,M$, are locally
supported. Hence, the construction of $\Smt^\Pl \varphi_i$ involves
the solution of a limited number of local problems
\eqref{local-problem-div} and we have $\Supp(\Smt^\Pl \varphi_i)
\subseteq \Supp(\varphi_i)$. Moreover, thanks to the local
characterization of the discrete divergence
\eqref{divergence-disc-Pl-Pl-2}, the entire computation of $\Smt^\Pl
\varphi_i$ requires $O(1)$ operations. This entails that the bilinear
forms $\Formveldisc^\Pl (\varphi_i, \varphi_j)$ and $\int_\Domain
\psi_k \Div \varphi_i$ and the linear form $\langle f, \Smt^\Pl
\varphi_i \rangle$ can be evaluated with $O(1)$ operations for all
$i,j=1, \dots, N$ and $k=1,\dots, M$. Thus, the discretization
\eqref{Stokes-Pl-Pl-2} is computationally feasible, in the sense of
Remark~\ref{R:computational-feasibility}. Let us mention also that the
stiffness matrices associated with $\Formveldisc^\Pl$ and its
counterpart in~\eqref{Stokes-Pl-Pl-2-standard} are of course different
but, for all $\eta>1$, their condition numbers differ, at most, by the
ratio of the continuity and the coercivity constants of
$\Formveldisc^\Pl$. This ratio is bounded by $c \eta^2(\eta-1)^{-1}$, as a consequence of Proposition~\ref{P:Pl-Pl-2-smoother} and Lemma~\ref{L:Pl-Pl-2-coercivity}.
\end{remark}

The following remarks connect~\eqref{Stokes-Pl-Pl-2} with other
existing discretizations.  

\begin{remark}[Connection with augmented Lagrangian formulations]
\label{R:connection-augmented-lagrangian}
In view of \eqref{Pl-Pl-2-smoother-stability}, the last summand $\eta
\int_\Domain \Grad \Rightinv_h^\Pl w_h \colon \Grad \Rightinv_h^\Pl
v_h$ in the definition of $\Formveldisc^\Pl$ penalizes the functions
that are in the discrete kernel $\SPkerdisc^{\Pl}$ and not in
$\SPker$. More precisely, the penalization is equivalent to $\eta
\int_\Domain \Div w_h \Div v_h$ on $\SPkerdisc^{\Pl}$. This indicates
that~\eqref{Stokes-Pl-Pl-2} can be interpreted as a new augmented
Lagrangian formulation for the Stokes problem; see
\cite[Section~6.1]{Boffi:Brezzi:Fortin.13}. The additional terms
enforcing consistency and symmetry distinguish our formulation from
previous ones. 
\end{remark}

\begin{remark}[Connection with DG discretizations]
\label{R:connection-DG}
The DG-SIP bilinear form in~\cite{Arnold:82} consists of four
terms. The first two terms serve to accommodate consistency, see
\cite[Section~4.2]{DiPietro.Ern:12} or~\cite{Veeser.Zanotti:18b}. In
particular, the second one arises due to the use of possibly
nonconforming, i.e. discontinuous, functions. The two remaining terms
are designed to further enforce symmetry and
coercivity, respectively, still preserving consistency. The same
structure can be observed in the 
form $\Formveldisc^\Pl$. Here nonconformity has to be intended in the
sense that $\SPkerdisc^\Pl \nsubseteq \SPker$, i.e. discretely
divergence-free functions are possibly not divergence-free. A
remarkable difference from the DG-SIP bilinear form is that the
coercivity of $\Formveldisc^\Pl$ can be guaranteed for all $\eta > 1$
and not only for sufficiently large $\eta$.  
\end{remark}

\begin{remark}[Connection with R-FEM discretizations]
\label{R:connection-recovered}
Rearranging terms in~\eqref{Pl-Pl-2-bilinear-form}, we see that the form $\Formveldisc^\Pl$ can be rewritten as follows
\begin{equation}
\label{Pl-Pl-2-bilinear-form-revisited}
\Formveldisc^\Pl (w_h, v_h)
=
\int_\Domain \Grad \Smt^\Pl w_h \colon \Grad \Smt^\Pl v_h
+
(\eta-1) \int_\Domain \Grad \Rightinv_h^\Pl w_h \colon \Grad \Rightinv_h^\Pl v_h.
\end{equation}
This sheds additional light on the condition $\eta>1$ in Lemma~\ref{L:Pl-Pl-2-coercivity} and provides an interesting connection with the Recovered Finite Element
Method (R-FEM) of Georgoulis and Pryer~\cite{Georgoulis.Pryer:18}.
\end{remark}

\subsection{Error estimates}
\label{SS:error-estimates}

We now aim at showing that, unlike
\eqref{Stokes-Pl-Pl-2-standard},
\eqref{Stokes-Pl-Pl-2} is a quasi-optimal
and pressure robust discretization of~\eqref{Stokes-weak}. As a preliminary step, we bound the
consistency error generated by the last two terms in the definition of
$\Formveldisc^\Pl$. Such terms can be expected to generate a
consistency error, as they were artificially added to the right-hand
side of~\eqref{Pl-Pl-2-bilinear-form-restriction}. 

\begin{lemma}[Consistency error]
\label{L:Pl-Pl-2-consistency}
Let $\eta > 1$ be given. We have
\begin{equation}
\label{Pl-Pl-2-consistency}
\NormSemi{\int_\Domain \Grad z_h \colon \Grad \Smt^\Pl v_h - \Formveldisc^\Pl (z_h, v_h)}
\Cleq 
\eta 
\inf_{z \in \SPker} \NormLeb{\Grad(z-z_h)}{} \NormLeb{\Grad v_h}{}
\end{equation}
for all $z_h \in \SPkerdisc^{\Pl}$  and $v_h \in (\Polybnd{\Degree}{1})^2$.
\end{lemma} 

\begin{proof}
The definitions of $\Formveldisc^\Pl$ and $\Smt^\Pl$ imply
\begin{equation*}
\int_\Domain \Grad z_h \colon \Grad \Smt^\Pl v_h - \Formveldisc^\Pl (z_h, v_h)
=
-\int_\Domain \Grad \Rightinv_h^\Pl z_h \colon \Grad(v_h + \eta \Rightinv_h^\Pl v_h).
\end{equation*}	
The equivalence~\eqref{Pl-Pl-2-smoother-stability} reveals, in
particular, $\NormLeb{\Grad \Rightinv_h^\Pl z_h}{} \Cleq
\NormLeb{\Grad(z-z_h)}{}$ for all $z \in \SPker$. The characterization
\eqref{divergence-disc-Pl-Pl-2} of the discrete divergence $\Divdisc$ and~\eqref{Pl-Pl-2-smoother-stability} entail also $\NormLeb{\Grad(v_h + \eta \Rightinv_h^\Pl v_h)}{} \Cleq \eta
\NormLeb{\Grad v_h}{}$. Inserting these bounds into the identity
above concludes the proof. 
\end{proof}

Recall from section~\ref{SS:quasi-optimality-press-robustness} that
the discrete velocity $u_h$ solving~\eqref{Stokes-Pl-Pl-2} is in the
discrete kernel $\SPkerdisc^{\Pl}$ and can be equivalently
characterized through the reduced problem 
\begin{equation}
\label{Stokes-Pl-Pl-2-reduced}
\forall z_h \in \SPkerdisc^{\Pl} \qquad
\Visc \,\Formveldisc^\Pl (u_h, z_h) = \langle f, \Smt^\Pl z_h \rangle. 
\end{equation}

\begin{theorem}[Quasi-optimality and pressure robustness]
\label{T:Pl-Pl-2-velocity-error}
For all $\eta >1$, problem~\eqref{Stokes-Pl-Pl-2} is a quasi-optimal
and pressure robust discretization of~\eqref{Stokes-weak} with
constant $\Cqopr \leq c \eta^2(\eta-1)^{-1}$. 
\end{theorem}

\begin{proof}
Denote by $u \in \SPker$ and $u_h \in \SPkerdisc^{\Pl}$ the solutions
of problems~\eqref{Stokes-reduced} and~\eqref{Stokes-Pl-Pl-2-reduced},
respectively, with load $f \in \SobD{}^2$ and viscosity $\Visc >
0$. Let $z_h \in \SPkerdisc^{\Pl}$ be arbitrary and define $v_h := u_h
- z_h$. Lemma~\ref{L:Pl-Pl-2-coercivity} and problem
\eqref{Stokes-Pl-Pl-2-reduced} reveal 
\begin{equation*}
\left( 1 - \dfrac{1}{\eta} \right) 
\NormLeb{\Grad(u_h- z_h)}{}^2
\leq
\dfrac{1}{\Visc} \langle f, \Smt^\Pl v_h \rangle 
- \Formveldisc^\Pl(z_h, v_h).
\end{equation*}
Since $v_h \in \SPkerdisc^{\Pl}$, we have $\Smt^\Pl v_h \in \SPker$ as
a consequence of Proposition~\ref{P:Pl-Pl-2-smoother}. Hence, problem
\eqref{Stokes-reduced} yields $\Visc^{-1} \langle f, \Smt^\Pl v_h
\rangle = \int_\Domain \Grad u \colon \Grad \Smt^\Pl v_h$. We insert
this identity into the previous inequality and invoke
Proposition~\ref{P:Pl-Pl-2-smoother} and
Lemma~\ref{L:Pl-Pl-2-consistency}. Owing to the inclusion $u \in
\SPker$, it results 
\begin{equation*}
\NormLeb{\Grad(u_h- z_h)}{}
\leq c \eta^2(\eta-1)^{-1} \NormLeb{\Grad(u-z_h)}{}.
\end{equation*} 
We conclude taking the infimum over all $z_h \in \SPkerdisc$ and
recalling~\eqref{best-errors-Vh-Zh}. 
\end{proof}

Let us mention that a better bound of the constant $\Cqopr$ in terms of $\eta$, namely $\Cqopr \leq c \eta (\eta-1)^{-1/2}$, could be obtained with the help of~\cite[Theorem~4.14]{Veeser.Zanotti:18}. Both, this estimate and the one in Theorem~\ref{T:Pl-Pl-2-velocity-error}, suggest to set $\eta=2$. The next remark additionally confirm that we may have $\Cqopr \to +\infty$ as $\eta \to +\infty$, thus pointing out the importance of explicitly knowing a safe value of the penalty parameter. 

\begin{remark}[Locking effect]
\label{R:locking}
The penalization in $\Formveldisc^\Pl$ imposes that the solution $u_h^\Pl$ of~\eqref{Stokes-Pl-Pl-2-reduced} approaches the subspace $\SPker \cap \SPkerdisc^\Pl$ for $\eta \to + \infty$, as a consequence of Proposition~\ref{P:Pl-Pl-2-smoother}. This entails that the constant $\Cqopr$ in Theorem~\ref{T:Pl-Pl-2-velocity-error} remains bounded in the limit $\eta \to +\infty$ only if the equivalence 
\begin{equation}
\label{best-errors-Pl-Pl-2}
\inf_{z_h \in \SPker \cap \SPkerdisc^{\Pl}} \NormLeb{\Grad(z-z_h)}{} \stackrel{!}{\eqsim}
\inf_{w_h \in (\Polybnd{\Degree}{1})^2} \NormLeb{\Grad(z-w_h)}{}
\end{equation}
holds for all $z \in \SPker$. Conversely, if~\eqref{best-errors-Pl-Pl-2} holds, we can assume that the function $z_h$ in the proof of Theorem~\ref{T:Pl-Pl-2-velocity-error} varies only in $\SPker \cap \SPkerdisc^\Pl$. This, in turn, provides a robust upper bound of $\Cqopr$ in the limit $\eta \to +\infty$. Whenever condition~\eqref{best-errors-Pl-Pl-2} fails, a locking effect may occur, in the sense of~\cite{Babuska.Suri:92b}. We illustrate this in section~\ref{SS:numerics-locking} by means of a numerical experiment.   
\end{remark} 

Theorem~\ref{T:Pl-Pl-2-velocity-error} states that the discretization
\eqref{Stokes-Pl-Pl-2} enjoys a better velocity $H^1$-error estimate than the standard one~\eqref{Stokes-Pl-Pl-2-standard},
cf. Remark~\ref{R:failure-identity}. The next result additionally
ensures that the two discretizations are actually comparable if one
considers the sum of the velocity $H^1$-error times viscosity plus the
pressure $L^2$-error. Thus, in other words, the modifications introduced in~\eqref{Stokes-Pl-Pl-2} do not impair the quasi-optimality of~\eqref{Stokes-Pl-Pl-2-standard}.  

\begin{theorem}[Quasi-optimality]
\label{T:Pl-Pl-2-pressure-error}
For all $\eta > 1$, problem~\eqref{Stokes-Pl-Pl-2} is a quasi-optimal
discretization of~\eqref{Stokes-weak} with constant $\Cqo \Cleq 
\eta^3/(\eta-1)$.
\end{theorem} 

\begin{proof}
Denote by $(u,p)$ and $(u_h, p_h)$ the solutions of problems
\eqref{Stokes-weak} and~\eqref{Stokes-Pl-Pl-2}, respectively, with
load $f \in \SobD{}^2$ and viscosity $\Visc>0$. In view of
Theorem~\ref{T:Pl-Pl-2-velocity-error}, it suffices to bound the
pressure error $\NormLeb{p-p_h}{}$. To this end, let $q_h \in
\Polyavg{\Degree-2}{0}$ be arbitrary and recall that the discrete
divergence $\Divdisc$ is given by~\eqref{divergence-disc}. The inf-sup
stability of the $\Poly{\Degree}/\Poly{\Degree-2}$ pair and
Proposition~\ref{P:Pl-Pl-2-smoother} yield 
\begin{equation*}
\NormLeb{p_h- q_h}{} \leq c
\sup_{v_h \in (\Polybnd{\Degree}{1})^2} \dfrac{\int_\Domain (p_h -
  q_h) \Div \Smt^\Pl v_h}{\NormLeb{\Grad v_h}{}}. 
\end{equation*}  
For all $v_h \in (\Polybnd{\Degree}{1})^2$, a comparison of
\eqref{Stokes-weak} and~\eqref{Stokes-Pl-Pl-2} entails 
\begin{equation*}
\int_\Domain (p_h - q_h) \Div \Smt^\Pl v_h
=
\Visc \left( \Formveldisc^\Pl(u_h, v_h) - \int_\Domain \Grad u \colon
  \Grad \Smt^\Pl v_h \right) + \int_\Domain (p-q_h) \Divdisc v_h 
\end{equation*}	
where we have made use again of
Proposition~\ref{P:Pl-Pl-2-smoother}. The last summand in the
right-hand side vanishes if we let $q_h$ be the $L^2$-orthogonal
projection of $p$. Hence, invoking Lemma~\ref{L:Pl-Pl-2-consistency}
and proceeding as in the proof of
Theorem~\ref{T:Pl-Pl-2-velocity-error}, we infer 
\begin{equation}
\label{Pl-Pl-2-pressure-error}
\NormLeb{p_h-q_h}{} \leq c \Visc \eta \NormLeb{\Grad(u-u_h)}{}.
\end{equation}
The triangle inequality and Theorem~\ref{T:Pl-Pl-2-velocity-error} conclude the proof.
\end{proof}

\subsection{Inhomogeneous continuity equation} 
\label{SS:inhomogeneous-continuity}

It is worth having a look at the case when the incompressibility
constraint $\Div u = 0$ of~\eqref{Stokes-strong} is replaced by the
inhomogeneous continuity condition $\Div u = g$ with $g \in
\LebH{}$. The corresponding weak formulation reads as follows: Find
$u \in \SobH{}^2$ and $p \in \LebH{}$ such that 
\begin{equation}
\label{Stokes-inhom}
\begin{alignedat}{2}
&\forall v \in \SobH{}^2
&\qquad
\Visc \int_\Domain \Grad u \colon \Grad v 
- \int_\Domain p \Div v 
&= \left\langle  f , v \right\rangle  \\
&\forall q \in \SPpres
&\qquad
\int_\Domain q \Div u &= \int_\Domain q g.
\end{alignedat}
\end{equation}

A possible extension of the discretization~\eqref{Stokes-Pl-Pl-2} with the $\Poly{\Degree}/\Poly{\Degree-2}$ pair consists in finding $u_h \in (\Polybnd{\Degree}{1})^2$ and $p_h \in \Polyavg{\Degree-2}{0}$ such that  
\begin{equation}
\label{Stokes-inhom-Pl-Pl-2}
\begin{aligned}
&\forall v_h \in (\Polybnd{\Degree}{1})^2
&\qquad
\Visc \,\Formveldisc^\Pl(u_h, v_h)
-\int_\Domain p_h \Div v_h 
&= \langle  f , \Smt^\Pl v_h \rangle  \\
&\forall q_h \in \Polyavg{\Degree-2}{0}
&\qquad
\int_\Domain q_h \Div u_h &= \int_\Domain q_h g.
\end{aligned}
\end{equation}

The second equations of~\eqref{Stokes-inhom} and
\eqref{Stokes-inhom-Pl-Pl-2} impose $u \in \SPker(g)$ and $u_h \in
\SPkerdisc^{\Pl}(g)$, respectively, where 
\begin{equation*}
\SPker(g) := \{ z \in \SobH{}^2 \mid \Div z = g \},
\qquad
\SPkerdisc^{\Pl}(g) := \{ z_h \in (\Polybnd{\Degree}{1})^2 \mid \Divdisc z_h = \Ritz_{\Degree-2} g  \} 
\end{equation*}
and $\Ritz_{\Degree-2}$ is the $L^2$-orthogonal projection onto
$\Polyavg{\Degree-2}{0}$.  

Lemma~\ref{L:Pl-Pl-2-consistency} states that the consistency error in
the left hand side of~\eqref{Pl-Pl-2-consistency} vanishes whenever
$z_h \in \SPker \cap \SPkerdisc^{\Pl}$. If, instead, we assume $z_h
\in \SPker(g) \cap \SPkerdisc^{\Pl}(g)$ for some $g \in \LebH{}$ with
$g \neq \Ritz_{\Degree-2} g$, the consistency error may not vanish. In fact, we possibly have $\Rightinv_h^\Pl z_h \neq 0$, as a
consequence of Proposition~\ref{P:Pl-Pl-2-smoother}. This suggests
that a bound of the consistency error solely in terms of the best
approximation $H^1$-error to $z_h$ by elements of $\SPker(g)$ is likely not possible. Therefore, we do not expect that the discrete velocity $u_h$ solving~\eqref{Stokes-inhom-Pl-Pl-2} is a near-best approximation of the analytical velocity in $(\Polybnd{\Degree}{1})^2$, with respect to the $H^1$-norm. 

Still, combining the equivalence~\eqref{Pl-Pl-2-smoother-stability} and the $L^2$-orthogonality of $\Ritz_{\Degree-2}$, we obtain the following generalization of Lemma~\ref{L:Pl-Pl-2-consistency} 
\begin{equation*}
\label{Pl-Pl-2-consistency-inhom}
\begin{split}
&\NormSemi{\int_\Domain \Grad z_h \colon \Grad \Smt^\Pl v_h - \Formveldisc^\Pl (z_h, v_h)} \leq \\
&\hspace{1.5cm}\leq c \eta 
\left( \inf_{z \in \SPker(g)} \NormLeb{\Grad(z-z_h)}{}
+ \inf_{q_h \in \Polyavg{\Degree-2}{0}} \NormLeb{g-q_h}{} \right) \NormLeb{\Grad v_h}{}
\end{split}
\end{equation*}
for all $z_h \in \SPkerdisc^{\Pl}(g)$ and $v_h \in (\Polybnd{\Degree}{1})^2$, with $g \in \LebH{}$.
%
Apart from the additional term in the right-hand side of this estimate,
the technique in the proof of Theorem~\ref{T:Pl-Pl-2-velocity-error} can be still applied, with the help of~\cite[Proposition~5.1.3]{Boffi:Brezzi:Fortin.13}, and we finally derive 
\begin{equation}
\label{Pl-Pl-2-velocity-error-inhom}
\NormLeb{\Grad(u-u_h)}{} \Cleq 
\inf_{v_h \in (\Polybnd{\Degree}{1})^2} \NormLeb{\Grad(u-v_h)}{} 
+ \inf_{q_h \in \Polyavg{\Degree-2}{0}} \NormLeb{g-q_h}{}
\end{equation} 
for any fixed $\eta >1$. Similarly as in~\eqref{intro:conforming-est}, here
the approximation power of the discrete pressure space in the $L^2$-norm may impair the velocity $H^1$-error, because the $\Poly{\Degree}/\Poly{\Degree-2}$ pair is unbalanced. We confirm this suspicion by means of a numerical experiment in section~\ref{SS:numerics-inhomogeneous}. Still, we remark that this estimate, unlike~\eqref{intro:conforming-est}, is pressure robust,
i.e. independent of the analytical pressure. A corresponding bound of
the pressure error can be derived arguing as in the proof
of Theorem~\ref{T:Pl-Pl-2-pressure-error}.  

The nonconforming discretization proposed in
section~\ref{SS:nonconforming-pairs} has the remarkable property that
the consistency error can always be bounded solely in terms of the
best approximation $H^1$-error to the analytical velocity; cf.
Remark~\ref{R:CR-inhom-continuity}. Therefore, in that case, we achieve
quasi-optimality and pressure robustness even if an inhomogeneous
continuity condition is imposed.

\section{Generalizations of the paradigmatic discretization} \label{S:generalizations}

The idea illustrated in the previous section can be generalized in
various directions. An immediate observation is that the same
construction applies to any other conforming and inf-sup stable pair
$\SPveldisc/\SPpresdisc$ such that 
\begin{itemize}
	\item[$(i)$] $\Polyavg{0}{0}$ is a subset of $\SPpresdisc$ and
	\item[$(ii)$] the discrete divergence $\Divdisc$ can be
          computed element-wise. 
\end{itemize} 
The first condition is needed in Proposition~\ref{P:Pl-Pl-2-smoother}
to ensure that the smoothing operator
$\Smt^\Pl$ fulfills~\eqref{conservation-divergence}. The second one
guarantees that the divergence correction $\Rightinv^\Pl$ can be computed element-wise. As a consequence, the proposed discretization is computationally
feasible, cf. Remark~\ref{R:Pl-Pl-2-feasibility}. 
%
Conditions (i) and (ii) are verified, for instance, by the following
generalization of the $\Poly{\Degree} / \Poly{\Degree-2}$ pair 
\begin{equation*}
\label{Pl-Pl-2-pair-generalized}
\SPveldisc = (\Polybnd{\Degree}{1})^\Dim
\qquad \text{and} \qquad
\SPpresdisc = \Polyavg{\Degree-k}{0},
\qquad
\Formpresdisc(v_h, q_h) = -\int_\Domain q_h \Div v_h
\end{equation*}
where $\Dim \leq k \leq \Degree$ and $\Dim \in \{2,3\}$. Another
possibility is to consider the conforming Crouzeix-Raviart pairs
described in~\cite[Sections~8.6.2 and
8.7.2]{Boffi:Brezzi:Fortin.13}. Stable pairs with continuous pressure,
i.e. $\SPpresdisc \subseteq C^0(\Domain)$, do not fulfill (i), while
(ii) is violated, for instance, by the modified Hood-Taylor pairs of
Boffi et al.~\cite{Boffi.Cavallini.Gardini.Gastaldi:12}.  

We now aim at addressing more substantial generalizations. We mainly
focus on the necessary modifications and, in particular, we omit all
proofs that are similar to the ones in the previous section. 

\subsection{Nonconforming pairs}
\label{SS:nonconforming-pairs}

Assume that $\SPveldisc/\SPpresdisc$ is a nonconforming pair,
i.e. $\SPveldisc \nsubseteq \SPvel$. In this case, it does not seem appropriate to define the smoothing
operator $\Smt$ as in~\eqref{Pl-Pl-2-smoother}, because of the condition $\Smt(\SPveldisc) \subseteq \SobH{\Domain}^\Dim$. A possible fix for this problem is to replace $v_h $
with $\Smtm v_h$, where $\Smtm: \SPveldisc \to \SPvel$ is a linear operator. To make sure that a counterpart of Proposition~\ref{P:Pl-Pl-2-smoother} holds, we require that $\Div \Smtm v_h$ has element-wise
the same mean as $\Divdisc v_h$ for all $v_h \in \SPveldisc$. Therefore, we resort to a element-wise
"mean mass preserving" operator; cf. Proposition~\ref{P:CR-mean-operator}. 

As before, we illustrate this idea by means of a model example, namely
the two-dimensional nonconforming Crouzeix-Raviart pair of degree
$\Degree \geq 2$. We do not consider the lowest-order case
$\Degree=1$, as it is rather specific and it is already covered by
\cite{Verfuerth.Zanotti:18}, cf. Remark~\ref{R:CR-first-order}. A
similar technique can be applied, for instance, with the modified
Crouzeix-Raviart pairs of~\cite{Matthies.Tobiska:05} or with the
three-dimensional generalizations of the Kouhia-Stenberg pair from 
\cite{Hu.Schedensack:18}. The original two-dimensional pair of Kouhia 
and Stenberg~\cite{Kouhia.Stenberg:95} can be treated as indicated in
Remark~\ref{R:CR-first-order}.  

Let the mesh $\Mesh$ be as in section~\ref{S:paradigmatic-discretization} and denote by $\Faces{}$ the faces of $\Mesh$. A subscript to $\Faces{}$ indicates that we consider only those faces that are contained in the set specified by the subscript. We orient each interior face $F \in
\Faces{\Domain}$ with a normal unit vector $\Normal_F$. We denote by
$\Jump{\cdot}_{|F}$ the jump on $F$ in the direction of
$\Normal_F$. For boundary faces $F \in \Faces{\partial \Domain}$, we
orient $\Normal_F$ so that it points outside $\Domain$ and let
$\Jump{\cdot}_{|F}$ coincide with the trace on $F$,
cf.~\cite[Section~1.2.3]{DiPietro.Ern:12}. 
We use the subscript $\Mesh$ to indicate the broken version of a
differential operator on $\Mesh$. For instance, the broken gradient of
an element-wise $H^1$-function $v$ is given by $(\GradM v)_{|K} :=
\Grad(v_{|K})$ for all $K \in \Mesh$. 

The nonconforming Crouzeix-Raviart space of degree $\Degree \in \N$ on $\Mesh$, with homogeneous boundary conditions, can be defined as follows 
\begin{equation*}
\label{CR-space}
\CR{\Degree} :=
\{ v \in \Polypiec{\Degree}{0}  \mid \forall F \in \Faces{}~\text{and}~ r \in \Poly{\Degree-1}(F) \quad \int_F \Jump{v} r = 0 \}.
\end{equation*}
Notice that the integral $\int_F v$ is well-defined for all $v \in
\CR{\Degree}$ and $F \in \Faces{}$ and vanishes if $F \in
\Faces{\partial \Domain}$. Yet, the jumps on mesh faces are
not vanishing in general. 

We assume hereafter $\Degree \geq 2$. The two-dimensional nonconforming
Crouziex-Raviart pair of degree $\Degree$ is 
\begin{equation*}
\label{CR-pair}
\SPveldisc = (\CR{\Degree})^2 
\qquad \text{and} \qquad
\SPpresdisc = \Polyavg{\Degree-1}{0},
\qquad
\Formpresdisc(v_h, q_h) = -\int_\Domain q_h \DivM v_h.
\end{equation*}
Results concerning the inf-sup stability can be found in
\cite{Baran.Stoyan:07,Crouzeix.Falk:89,Fortin.Soulie:83}. Since the broken divergence $\DivM$ maps $\SPveldisc$ into $\SPpresdisc$, it coincides with the discrete divergence from~\eqref{divergence-disc}, i.e. $\Divdisc = \DivM$. We measure the velocity error in the broken $H^1$-norm, augmented with scaled jumps. Thus, in the notation of section~\ref{S:abstract-framework}, we set 
\begin{equation*}
\label{CR-norm}
\Normh{v}^2 =\Normh[\Cr]{v}^2 :=
\NormLeb{\GradM v}{}^2 + \sum_{F \in \Faces{}} h_F^{-1}
\NormLeb{\Jump{v}}{F}^2,
\end{equation*} where $h_F$ is the diameter of $F$. An
equivalent alternative would be to consider only the broken
$H^1$-norm. Both options extend the $H^1$-norm to $\SobH{}^2 +
(\CR{\Degree})^2$.  

Let $\Nodes{\Degree, \Domain}$ be the set of interior Lagrange nodes
of degree $\Degree$ in $\Mesh$. For all $\nu \in \Nodes{\Degree,
  \Domain}$, we denote by $\Phi_\Degree^\nu$ the Lagrange basis
function of $\Polybnd{\Degree}{1}$ associated with the evaluation at
$\nu$, i.e. $\Phi_\Degree^\nu(\nu') = \delta_{\nu \nu'}$ for all $\nu'
\in \Nodes{\Degree, \Domain}$. Fix also an element $K_\nu \in \Mesh$
with $\nu \in K_\nu$. We define a "simplified nodal averaging"
operator $\Smtavg^\Cr: (\CR{\Degree})^2 \to (\Polybnd{\Degree}{1})^2$ 
by
\begin{equation*}
\label{CR-averaging}
\Smtavg^\Cr v_h :=
\sum_{\nu \in \Nodes{\Degree,\Domain}} v_{h|K_\nu}(\nu)\, \Phi_\Degree^\nu.
\end{equation*}   

Next, let $m_F$ be the midpoint of any interior face $F \in
\Faces{\Domain}$. Consider the bubble function $\Phi_2^F:= 3(2
\NormSemi{F})^{-1}\Phi_2^{m_F}$, where $\Phi_2^{m_F}$ is the Lagrange
basis function of $\Polybnd{2}{1}$ associated with the evaluation at
$m_F$. The normalization implies $\int_{F'} \Phi_2^F = \delta_{FF'}$
for all $F' \in \Faces{}$, according to the Simpson quadrature
formula. We introduce a "bubble" operator $\Smtfac^\Cr:
(\CR{\Degree})^2 \to (\Polybnd{\Degree}{1})^2$ by
\begin{equation*}
\label{CR-bubble-operator}
\Smtfac^\Cr v_h := \sum_{F \in \Faces{\Domain}} \left( \int_Fv_h \right) \Phi_2^F.
\end{equation*}

We combine $\Smtavg^\Cr$ and $\Smtfac^\Cr$ to obtain the announced
element-wise mean mass preserving operator $\Smtm^\Cr$. Roughly speaking, we use
$\Smtfac^\Cr$ to enforce the first part of
\eqref{CR-mean-operator-properties} below, while $\Smtavg^\Cr$ is
responsible for the second part. 

\begin{proposition}[Element-wise mean mass preserving operator]
\label{P:CR-mean-operator}
The linear operator $\Smtm^\Cr: (\CR{\Degree})^2 \to (\Polybnd{\Degree}{1})^2$ given by
\begin{equation}
\label{CR-mean-operator}
\Smtm^\Cr v_h := \Smtavg^\Cr v_h + \Smtfac^\Cr (v_h - \Smtavg^\Cr v_h)
\end{equation}
is such that
\begin{equation}
\label{CR-mean-operator-properties}
\int_K \Div \Smtm^\Cr v_h = \int_K \Div v_h
\quad \text{and} \quad
\Normh[\Cr]{v_h - \Smtm^\Cr v_h} \leq c \inf_{v \in \SobH{}^2} \Normh[\Cr]{v-v_h}
\end{equation}
for all $v_h \in (\CR{\Degree})^2$ and $K \in \Mesh$. 
\end{proposition} 

\begin{proof}
Let $v_h \in (\CR{\Degree})^2$ and $F' \in \Faces{\Domain}$ be given. The normalization of the functions $\{ \Phi_2^F \}_{F \in \Faces{\Domain}}$ reveals 
\begin{equation*}
\int_{F'} \Smtfac^\Cr (v_h - \Smtavg^\Cr v_h) =
\sum_{F \in \Faces{\Domain}}\int_F (v_h - \Smtavg^\Cr v_h) \delta_{FF'} =
\int_{F'} (v_h - \Smtavg^\Cr v_h).
\end{equation*}
The same identities hold also for boundary faces $F' \in \Faces{\partial \Domain}$, in view of the boundary conditions in $\CR{\Degree}$ and $\Polybnd{\Degree}{1}$. Rearranging terms, we obtain $\int_{F'} \Smtm^\Cr v_h =
\int_{F'} v_h$ for all $F' \in \Faces{}$. Then, for all $K \in \Mesh$, the Gauss theorem yields the first part of~\eqref{CR-mean-operator-properties}
\begin{equation*}
\int_K \Div \Smtm^\Cr v_h = 
\sum_{F' \in \Faces{\partial K}} \int_{F'} \Smtm^\Cr v_h \cdot \Normal_K =
\sum_{F' \in \Faces{\partial K}} \int_{F'} v_h \cdot \Normal_K =
\int_K \Div v_h.
\end{equation*}
A detailed proof of the second part of
\eqref{CR-mean-operator-properties} can be found in
\cite[Section~3]{Veeser.Zanotti:18b}, where a similar, actually more
involved, operator is considered. For this reason, we only sketch the
proof. Let $K \in \Mesh$ be given. Owing to the triangle inequality,
we initially bound $\NormLeb{\Grad(v_h - \Smtavg^\Cr v_h)}{K}$ and
$\NormLeb{\Grad\Smtfac^\Cr(v_h - \Smtavg^\Cr v_h)}{K}$. The scaling of
the functions $\{\Phi_2^F\}_{F \in \Faces{\partial K}}$ and the trace
inequality imply 
\begin{equation}
\label{CR-mean-operator-proof}
\begin{split}
\NormLeb{\Grad(v_h - \Smtavg^\Cr v_h)}{K} &+
\NormLeb{\Grad\Smtfac^\Cr(v_h - \Smtavg^\Cr v_h)}{K}
\\
&\Cleq h_K^{-1} \NormLeb{v_h - \Smtavg^\Cr v_h}{K} + 
\NormLeb{\Grad(v_h - \Smtavg^\Cr v_h)}{K},
\end{split}
\end{equation} 
where $h_K$ is the diameter of $K$. Next, for all $\nu \in
\Nodes{\Degree,K}$, we have $v_{h|K} (\nu) = \Smtavg^\Cr v_h(\nu)$ if
$\nu \in \mathrm{int}(K)$, otherwise $\NormSemi{v_{h|K} (\nu) -
  \Smtavg^\Cr v_h(\nu)} \Cleq \sum_{F \ni \nu} h_F^{-1/2}
\NormLeb{\Jump{v_h}}{F}$, where $F$ varies in $\Faces{}$. This
estimate and the scaling of the Lagrange basis functions entail that
the right-hand side of~\eqref{CR-mean-operator-proof} is bounded by
$\sum_{F \cap K \neq \emptyset} h_F^{-1/2}
\NormLeb{\Jump{v_h}}{F}$. Squaring and summing over all $K \in \Mesh$,
we finally obtain 
\begin{equation*}
\NormLeb{\GradM(v_h - \Smtm^\Cr v_h)}{}^2
\Cleq \sum_{F \in \Faces{}} h_F^{-1} \NormLeb{\Jump{v_h}}{F}^2.
\end{equation*}
We conclude recalling the definition of the norm $\Normh[\Cr]{\cdot}$.
\end{proof}

According to the first part of~\eqref{CR-mean-operator-properties}, we
can now construct a smoothing operator similarly to $\Smt^\Pl$ in
Proposition~\ref{P:Pl-Pl-2-smoother}. Recalling the local
operators $\Rightinv_\Degree^K$ introduced in section~\ref{SS:local-inversion-divergence},
we define $\Smt^\Cr: (\CR{\Degree})^2 \to \SobH{\Domain}^2$ by  
\begin{equation}
\label{CR-smoother}
\Smt^\Cr v_h := \Smtm^\Cr v_h + 
\sum_{K \in \Mesh} \Rightinv_\Degree^K (\DivM v_h - \Div \Smtm^\Cr v_h).
\end{equation}
Owing to the identity $\Divdisc = \DivM$, we see that $\Smt^\Cr$
fulfills condition~\eqref{conservation-divergence}, as a consequence of Propositions~\ref{P:local-right-inverse} and
\ref{P:CR-mean-operator}. Moreover, the stability of the operators
$\Rightinv_\Degree^K$ and the second part of
\eqref{CR-mean-operator-properties} provide a strengthened counterpart
of~\eqref{Pl-Pl-2-smoother-stability} in that, for all $v_h \in
(\CR{\Degree})^2$, we have 
\begin{equation}
\label{CR-smoother-stability}
\Normh[\Cr]{v_h - \Smt^\Cr v_h} \Cleq  \inf_{v \in \SobH{}^2} \Normh[\Cr]{v-v_h}. 
\end{equation}

Next, inspired by the definition of $\Formveldisc^\Pl$ in
\eqref{Pl-Pl-2-bilinear-form} as well as by identity~\eqref{Pl-Pl-2-bilinear-form-revisited}, we introduce the following
bilinear form $\Formveldisc^\Cr$ on $(\CR{\Degree})^2$  
\begin{equation*}
\Formveldisc^\Cr (w_h, v_h) :=
\int_\Domain \Grad \Smt^\Cr w_h \colon \Grad \Smt^\Cr v_h +
(\eta-1) \int_{\Domain} \GradM \Rightinv_h^\Cr w_h \colon \GradM \Rightinv_h^\Cr v_h
\end{equation*}
where $\Rightinv_h^\Cr := (\Smt^\Cr - \Id)$ and $\eta > 1$ is a penalty
parameter. The above-mentioned properties of $\Smt^\Cr$ imply that the
necessary conditions in Lemmas~\ref{L:quasi-optimal-nec} and
\ref{L:quasi-optimal-press-robust-nec} are fulfilled if we set
$\Formveldisc = \Formveldisc^\Cr$ and $\Smt = \Smt^\Cr$.  
%
In this setting, the abstract discretization~\eqref{Stokes-disc} reads as
follows: Find $u_h \in (\CR{\Degree})^2$ and $p_h \in
\Polyavg{\Degree-1}{0}$ such that 
\begin{equation}
\label{Stokes-CR}
\begin{alignedat}{2}
&\forall v_h \in (\CR{\Degree})^2
&\qquad
\Visc \,\Formveldisc^\Cr(u_h, v_h)
-\int_\Domain p_h \DivM v_h 
&= \left\langle  f , \Smt^\Cr v_h \right\rangle  \\
&\forall q_h \in \Polyavg{\Degree-1}{0}
&\qquad
\int_\Domain q_h \DivM u_h &= 0.
\end{alignedat}
\end{equation}

Similarly as $\Formveldisc^\Pl$ in Lemma~\ref{L:Pl-Pl-2-coercivity}, the form $\Formveldisc^\Cr$ is coercive on $(\CR{\Degree})^2$, for $\eta > 1$, with constant $\geq (1-\eta^{-1})$. Moreover, in view of~\eqref{CR-smoother-stability}, we can
estimate the consistency error of~\eqref{Stokes-CR} by the following counterpart of
Lemma~\ref{L:Pl-Pl-2-consistency} 
\begin{equation}
\label{CR-consistency}
\NormSemi{\int_\Domain \GradM w_h \colon \Grad \Smt^\Cr v_h - \Formveldisc^\Cr (w_h, v_h)}
\leq c \eta 
\inf_{w \in \SobH{}^2} \Normh[\Cr]{w-w_h} \Normh[\Cr]{v_h}
\end{equation}  
for all $w_h, v_h \in (\CR{\Degree})^2$. Hence, we conclude that
\eqref{Stokes-CR} is a quasi-optimal and pressure robust
discretization of~\eqref{Stokes-weak} in the norm
$\Normh[\Cr]{\cdot}$ and the constant $\Cqopr$ from Definition~\ref{D:quasi-optimality-press-robust} solely depends on $\eta$ and the shape parameter of $\Mesh$. 

Whenever the pair $(\CR{\Degree})^2/ \Polyavg{\Degree-1}{0}$ is inf-sup stable, an estimate of the pressure $L^2$-error, only in terms of the best approximation errors to the analytical velocity and the analytical pressure, can also be established similarly as in
Theorem~\ref{T:Pl-Pl-2-pressure-error}. Thus, problem
\eqref{Stokes-CR} is also a quasi-optimal discretization of
\eqref{Stokes-weak}. 

Locally supported basis functions of
$\CR{\Degree}$ are described in
\cite[section~3]{Baran.Stoyan:06}. With this basis and the standard
nodal basis of $\Polypiec{\Degree-1}{0}$, we see that
\eqref{Stokes-CR} is computationally feasible in the sense of
Remark~\ref{R:computational-feasibility},
cf. Remark~\ref{R:Pl-Pl-2-feasibility}. 

\begin{remark}[The pair $(\CR{1})^2/\Polyavg{0}{0}$]
\label{R:CR-first-order}
In principle, the approach described for $\Degree \geq 2$ applies also
with $\Degree = 1$, up to observing that $\Rightinv_2^K$ (and not $\Rightinv_1^K$) should be used in
\eqref{CR-smoother}. The point is that, in this case, an element-wise integration
by parts and the identity $\int_{F'} \Smt^\Cr v_h = \int_{F'}
\Smtm^\Cr v_h = \int_{F'} v_h$, with $F' \in \Faces{}$, reveal
$\int_\Domain \GradM w_h \colon \Grad \Rightinv_h^\Cr v_h = 0$ for all
$w_h, v_h \in (\CR{1})^2$. Hence, the form $\Formveldisc^\Cr$ is given by $\Formveldisc^\Cr(w_h, v_h) = \int_\Domain \GradM w_h
\colon \GradM v_h + \eta\int_\Domain \GradM \Rightinv_h^\Cr w_h \colon
\GradM \Rightinv_h^\Cr v_h$, showing that the penalization is actually
not needed. Setting $\eta=0$ annihilates the consistency error and
corresponds to the discretization proposed in
\cite{Verfuerth.Zanotti:18}.  
\end{remark}

\begin{remark}[Inhomogeneous continuity equation]
\label{R:CR-inhom-continuity}
The infimum in the right-hand side of~\eqref{CR-consistency} is taken
over $\SobH{}^2$ and not only over $\SPker$, unlike
Lemma~\ref{L:Pl-Pl-2-consistency}. This prevents the issue pointed out
in section~\ref{SS:inhomogeneous-continuity}. Therefore, the nonconforming
Crouzeix-Raviart pair can be used to design a quasi-optimal and
pressure robust discretization of problem~\eqref{Stokes-inhom} with
the inhomogeneous continuity condition $g \neq 0$. 	 
\end{remark}

\subsection{Conforming pairs with continuous pressure}
\label{SS:continuous-pressure}

Another class of pairs still not covered by our discussion are conforming pairs with continuous pressure. In fact, the following observations obstruct the construction of a smoothing operator as indicated in Proposition~\ref{P:Pl-Pl-2-smoother}.
\begin{itemize}
	\item[$(i)$] Since $\Polyavg{0}{0}$ is not a subspace of $\SPpresdisc$, the identity $\int_K \Divdisc v_h = \int_K \Div v_h$ may fail to hold for some $v_h \in \SPveldisc$ and $K \in \Mesh$.
	\item[$(ii)$] The computation of $\Divdisc$ is likely
          unfeasible in the sense of
          Remark~\ref{R:computational-feasibility}. 
\end{itemize}
Item (i) entails that we cannot correct the divergence element-wise by
means of the operators $\Rightinv_\Degree^K$ from
section~\ref{SS:local-inversion-divergence}. The shape functions of
the lowest-order continuous space $\Polypiec{1}{1}$ suggest to work on
patches of elements sharing a vertex, instead. Item (ii) further
indicates that we should never require a direct computation of
$\Divdisc$. The construction of a quasi-optimal and pressure robust
discretization of the Stokes equations is still possible under these
constraints, but it is more involved than the ones in the
previous sections. We mainly adapt ideas by Lederer 
et al.~\cite{Lederer.Linke.Merdon.Schoberl:17}. 

As an example, we let the mesh $\Mesh$ be as in section~\ref{S:paradigmatic-discretization} and consider the two-dimensional Hood-Taylor pair  
\begin{equation*}
\label{Hood-Taylor-pair}
\SPveldisc = (\Polybnd{\Degree}{1})^2 
\qquad \text{and} \qquad
\SPpresdisc = \Polyavg{\Degree-1}{1},
\qquad
\Formpresdisc(v_h, q_h) = -\int_\Domain q_h \Div v_h
\end{equation*}
with $\Degree \geq 2$. The inf-sup condition~\eqref{inf-sup-disc} holds with $\beta^{-1} \leq c$ under mild assumptions on $\Mesh$, see~\cite{Boffi:94}. The discrete divergence coincides with the $L^2$-orthogonal projection of the analytical divergence onto $\Polyavg{\Degree-1}{1}$. We denote by $\SPkerdisc^\Ht$ the discrete kernel.

Let $\Nodes{} := \Nodes{1}$ denote the set of all vertices of
$\Mesh$. For each $\nu \in \Nodes{}$, let $\Phi_1^\nu$ be the Lagrange
basis function of $\Polypiec{1}{1}$ associated with the evaluation at
$\nu$, i.e. $\Phi_1^\nu(\nu') = \delta_{\nu\nu'}$ for all $\nu' \in
\Nodes{}$. Recall that $\Phi_1^\nu$ is supported on the patch
$\omega_\nu := \{ K \in \Mesh \mid \nu \in K \}$. Consider the
barycentric refinement $\Submesh_\nu$ of $\omega_\nu$, i.e. the mesh
obtained connecting the vertices and the barycenter of any triangle in
$\omega_\nu$, cf. Figure~\ref{F:barycentric-refinement-star}. The
space $\Polypiec{\Degree}{0}(\Submesh_\nu)$ and the subspaces  
\begin{equation*}
\label{local-spaces-star}
\Polybnd{\Degree}{1}(\Submesh_\nu)
\qquad \text{and} \qquad
\Polyavg{\Degree-1}{0}(\Submesh_\nu).
\end{equation*}
are defined on $\Submesh_\nu$ analogously to $\Polypiec{\Degree}{0}$
in~\eqref{elementwise-polynomials} and $\Polybnd{\Degree}{1}$ and
$\Polyavg{\Degree-1}{0}$ in~\eqref{elementwise-polynomials-sub},
respectively. The element-wise local Lagrange interpolant
$\SmtLag_\Degree^\nu : \Polypiec{\Degree}{0}(\Submesh_\nu) \to
\Polypiec{\Degree-1}{0}(\Submesh_\nu)$ is given by  
\begin{equation*}
\label{Hood-Taylor-local-Lagrange}
\SmtLag_\Degree^\nu v := \sum_{K \in \Submesh_\nu}
\sum_{\nu' \in \Nodes{\Degree-1, K} } v_{|K} (\nu') \Phi_{\Degree-1}^{\nu',K}
\end{equation*}
where $\Nodes{\Degree-1, K}$ is the set of Lagrange nodes of degree
$\Degree-1$ in $K$ and $\Phi_{\Degree-1}^{\nu',K}$ is the Lagrange
basis function of $\Poly{\Degree}(K)$ associated with the evaluation
at $\nu'$ and extended to zero outside $K$. Consider also the
simplified local averaging $\Smtavgloc_\Degree^\nu:
\Polypiec{\Degree}{0}(\Submesh_\nu) \to \Polypiec{\Degree-1}{1}$ 
\begin{equation*}
\label{Hood-Taylor-local-averaging}
\Smtavgloc_\Degree^\nu v := 
\sum_{\nu \in \Nodes{\Degree-1}} v_{|{K_\nu}} (\nu) \Phi_{\Degree-1}^\nu
\end{equation*}
where $K_\nu \in \Mesh$ is a fixed element such that $\nu \in K_\nu$
and $v$ is extended to zero outside $\omega_\nu$. As before,
$\Phi_{\Degree-1}^\nu$ denotes the Lagrange basis function of
$\Polybnd{\Degree-1}{1}$ associated with the evaluation at $\nu$. 

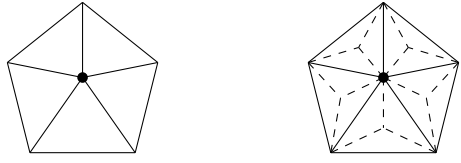
\begin{figure}[ht]
	\centering
	\begin{tikzpicture}
	\coordinate (z1) at (0.5,1.2);
	\coordinate (z2) at (0.8,0);
	\coordinate (z3) at (2.2,0);
	\coordinate (z4) at (2.5,1.2);
	\coordinate (z5) at (1.5,2);
	\coordinate (c) at  (1.5,1);
	\path (z1) edge (z2);
	\path (z2) edge (z3);
	\path (z3) edge (z4);
	\path (z4) edge (z5);
	\path (z5) edge (z1);
	\path (z1) edge (c);
	\path (z2) edge (c);
	\path (z3) edge (c);
	\path (z4) edge (c);
	\path (z5) edge (c);
	\fill (c) circle (2pt);
	\coordinate (w1) at (4.5,1.2);
	\coordinate (w2) at (4.8,0);
	\coordinate (w3) at (6.2,0);
	\coordinate (w4) at (6.5,1.2);
	\coordinate (w5) at (5.5,2);
	\coordinate (cc) at (5.5,1);
	\path (w1) edge (w2);
	\path (w2) edge (w3);
	\path (w3) edge (w4);
	\path (w4) edge (w5);
	\path (w5) edge (w1);
	\path (w1) edge (cc);
	\path (w2) edge (cc);
	\path (w3) edge (cc);
	\path (w4) edge (cc);
	\path (w5) edge (cc);
	\fill (cc) circle (2pt);
	\coordinate (b1) at (4.93,0.73);
	\coordinate (b2) at (5.5,0.33);
	\coordinate (b3) at (6.07,0.73);
	\coordinate (b4) at (5.83,1.4);
	\coordinate (b5) at (5.17,1.4);
	\path[dashed] (w1) edge (b1);
	\path[dashed] (w2) edge (b1);
	\path[dashed] (cc) edge (b1);
	\path[dashed] (w2) edge (b2);
	\path[dashed] (w3) edge (b2);
	\path[dashed] (cc) edge (b2);
	\path[dashed] (w3) edge (b3);
	\path[dashed] (w4) edge (b3);
	\path[dashed] (cc) edge (b3);
	\path[dashed] (w4) edge (b4);
	\path[dashed] (w5) edge (b4);
	\path[dashed] (cc) edge (b4);
	\path[dashed] (w5) edge (b5);
	\path[dashed] (w1) edge (b5);
	\path[dashed] (cc) edge (b5);
	\end{tikzpicture}
	\caption{Generic patch $\omega_\nu$ (left) and barycentric refinement $\Submesh_\nu$ (right).}
	\label{F:barycentric-refinement-star}
\end{figure}

We are now ready to define the operators $\Rightinv_\Degree^\nu: \Leb{} \to \SobH{}^2$ that will be used to correct the divergence in each patch $\omega_\nu$, $\nu \in \Nodes{}$. Here $\Rightinv_\Degree^\nu$ plays the same role as $\Rightinv_\Degree^K$ in section~\ref{SS:local-inversion-divergence}. Given $q \in \Leb{}$, let $u_\nu = u_\nu(q) \in \Polybnd{\Degree}{1}(\Submesh_\nu)^2$ and $p_\nu = p_\nu(q) \in \Polyavg{\Degree-1}{0}(\Submesh_\nu)$ be such that
\begin{equation}
\label{local-problem-div-star}
\begin{aligned}
&\forall v_\nu \in \Polybnd{\Degree}{1}(\Submesh_\nu)^2
\qquad
\int_{\omega_\nu} \Grad u_\nu \colon \Grad v_\nu 
- \int_{\omega_\nu} p_\nu \Div v_\nu = 0\\
&\forall q_\nu \in \Polyavg{\Degree-1}{0}(\Submesh_\nu)
\quad \;
\int_{\omega_\nu} q_\nu \Div u_\nu = 
\int_{\omega_\nu} \left ( \Smtavgloc_\Degree^\nu(q_\nu \Phi_1^\nu) - \SmtLag_\Degree^\nu(q_\nu \Phi_1^\nu) \right )q.
\end{aligned}
\end{equation} 
This problem is uniquely solvable, according to~\cite[Corollary~6.2]{Guzman.Neilan:18}. Then, we set
\begin{equation*}
\label{local-right-inverse-star}
\Rightinv_\Degree^{\nu} q := u_\nu \quad \text{in} \;\; \omega_\nu
\qquad \text{and} \qquad
\Rightinv_\Degree^{\nu} q := 0 \quad \text{in} \;\; \Domain \setminus \omega_\nu.
\end{equation*}

\begin{remark}[Local problems]
\label{R:local-problem-star}
The use of the barycentric refinement $\Submesh_\nu$ is a main
difference compared to~\cite{Lederer.Linke.Merdon.Schoberl:17}. This 
ensures that the pair
$\Polybnd{\Degree}{1}(\Submesh_\nu)^2/
\Polyavg{\Degree-1}{0}(\Submesh_\nu)$ is inf-sup stable. In fact, it
is known that the stability of the Scott-Vogelius pair on $\omega_\nu$ (without the barycentric
refinement) may be impaired if $\nu$ is a singular or nearly singular
vertex, see~\cite{Scott.Vogelius:85}. The partition of unity $\{
\Phi_1^\nu \}_{\nu \in \Nodes{}}$ and the interpolants $\{
\SmtLag_\Degree^\nu \}_{\nu \in \Nodes{}}$ account for the overlapping
of the patches, while the averaging operators $\{
\Smtavgloc_\Degree^\nu \}_{\nu \in \Nodes{}}$ are used to avoid a
direct computation of the discrete divergence
in~\eqref{Hood-Taylor-divergence-correction}. 
\end{remark}

We define a global divergence correction $\Rightinv_h^\Ht: (\Polybnd{\Degree}{1})^2 \to \SobH{}^2$ 
\begin{equation}
\label{Hood-Taylor-divergence-correction}
\Rightinv_h^\Ht v_h := 
\sum_{\nu \in \Nodes{} } \Rightinv_\Degree^\nu \Div v_h.
\end{equation}
In contrast to $\Smt^\Pl$ and $\Smt^\Cr$ from~\eqref{Pl-Pl-2-smoother}
and~\eqref{CR-smoother}, respectively, we now make use of a smoothing
operator $\Smt^\Ht$ which is not guaranteed to be
divergence-preserving, i.e.~\eqref{conservation-divergence} may fail
to hold. We shall see, however, that it still
satisfies the necessary conditions in
Lemmas~\ref{L:quasi-optimal-nec}
and~\ref{L:quasi-optimal-press-robust-nec}. In the following
proposition we only prove a basic stability estimate, for the sake of
simplicity. 

\begin{proposition}[Smoothing operator for the Hood-Taylor pair]
\label{P:Hood-Taylor-smoother}
The linear operator $\Smt^\Ht: (\Polybnd{\Degree}{1})^2 \to \SobH{}^2$ given by 
\begin{equation*}
\label{Hood-Taylor-smoother}
\Smt^\Ht v_h := v_h + \Rightinv_h^\Ht v_h
\end{equation*}
satisfies~\eqref{quasi-optimal-nec-div} and~\eqref{quasi-optimal-press-robust-nec-div} and is such that, for all $v_h \in (\Polybnd{\Degree}{1})^2$,
\begin{equation}
\label{Hood-Taylor-smoother-stability}
\NormLeb{\Grad(v_h - \Smt^\Ht v_h)}{}
\leq c \NormLeb{\Div v_h}{}.
\end{equation}
\end{proposition}

\begin{proof}
For all $v_h \in (\Polybnd{\Degree}{1})^2$ and $q_h \in \Polyavg{\Degree-1}{1}$, we have
\begin{equation*}
\int_\Domain q_h \Div \Rightinv_h^\Ht v_h =
\sum_{\nu \in \Nodes{}} \int_{\omega_\nu}
( \Smtavgloc_\Degree^\nu (q_h \Phi_1^\nu) - \SmtLag_\Degree^\nu (q_h \Phi_1^\nu) )\Div v_h = 0.
\end{equation*}
The first identity follows from the second equation of
\eqref{local-problem-div-star}, which actually holds for all $q_\nu$
in $\Polypiec{\Degree-1}{0}(\Submesh_\nu)$ (and not only in
$\Polyavg{\Degree-1}{0}(\Submesh_\nu)$), as both sides vanish if
$q_\nu$ is constant. To check the second identity, observe that
$\Smtavgloc_\Degree^\nu (q_h \Phi_1^\nu) = \SmtLag_\Degree^\nu (q_h
\Phi_1^\nu)$ for all $\nu \in \Nodes{}$, due to the continuity of $q_h
\Phi_1^\nu$. Thus, we derive the identity 
\begin{equation*}
\int_\Domain q_h \Div \Smt^\Ht v_h = \int_\Domain q_h \Div v_h
\end{equation*}
showing that condition~\eqref{quasi-optimal-nec-div} holds. Next, let
$z_h \in \SPker^\Ht$ be given and consider $q_h = \Div \Smt^\Ht z_h
$. Recall that $\{ \Phi_1^\nu \}_{\nu \in \Nodes{}}$ is a partition of
unity and extend $\SmtLag_\Degree^\nu(q_h \Phi_1^\nu)$ to zero outside
$\omega_\nu$. We infer $\sum_{\nu \in \Nodes{}}
\SmtLag_\Degree^\nu(q_h \Phi_1^\nu) = q_h$. 
Then, since $z_h$ is discretely divergence-free, we have  
\begin{equation*}
\NormLeb{q_h}{}^2 =
\int_\Domain q_h \Div z_h -  \sum_{\nu \in \Nodes{}} \int_\Domain \SmtLag_\Degree^\nu(q_h \Phi_1^\nu) \Div z_h = 0.
\end{equation*}
This reveals $\Div \Smt^\Ht z_h = 0$ and confirms that condition~\eqref{quasi-optimal-press-robust-nec-div} holds. Finally, owing to the stability of $\SmtLag_\Degree^\nu$ and $\Smtavgloc_\Degree^\nu$ in the $L^2(\omega_\nu)$-norm, we infer 
\begin{equation*}
\sup_{q_\nu \in \Polyavg{\Degree-1}{0}(\Submesh_\nu)}
\dfrac{\int_{\omega_\nu} \left ( \Smtavgloc_\Degree^\nu(q_\nu \Phi_1^\nu) - \SmtLag_\Degree^\nu(q_\nu \Phi_1^\nu) \right ) \Div v_h}{\NormLeb{q_\nu}{{\omega_\nu}}}
\leq c \NormLeb{\Div v_h}{\omega_\nu}
\end{equation*}
for all $\nu \in \Nodes{}$ and $v_h \in
(\Polybnd{\Degree}{1})^2$. This entails $\NormLeb{\Grad
  \Div \Rightinv_\Degree^\nu v_h}{\omega_\nu} \Cleq \NormLeb{\Div
  v_h}{\omega_\nu}$, owing to
\cite[Corollary~4.2.1]{Boffi:Brezzi:Fortin.13} and the inf-sup
stability of the pair $\Polybnd{\Degree}{1}(\Submesh_\nu)^2/
\Polyavg{\Degree-1}{0}(\Submesh_\nu)$ stated in
\cite[Corollary~6.2]{Guzman.Neilan:18}. The definition of
$\Rightinv_h^\Ht$ in~\eqref{Hood-Taylor-divergence-correction} then
implies 
\begin{equation*}
\NormLeb{\Grad \Rightinv_h^\Ht v_h}{K} 
\Cleq
\sum_{K' \cap K \neq \emptyset} \NormLeb{\Div v_h}{K'} 
\end{equation*}
for all $K \in \Mesh$, where $K'$ varies in $\Mesh$. We conclude summing over all elements of $\Mesh$ and recalling the definition of $\Smt^\Ht$.  
\end{proof}

Next, for $\eta > 1$, we introduce the following bilinear form on $(\Polybnd{\Degree}{1})^2$
\begin{equation*}
\label{Hood-Taylor-bilinear-form}
\Formveldisc^\Ht(w_h, v_h) :=
\int_\Domain \Grad \Smt^\Ht w_h \colon \Grad \Smt^\Ht v_h +
(\eta-1) \int_\Domain \Grad \Rightinv_h^\Ht w_h \colon \Grad \Rightinv_h^\Ht v_h.
\end{equation*}
The abstract discretization~\eqref{Stokes-disc} with $\Formveldisc = \Formveldisc^\Ht$ and $\Smt = \Smt^\Ht$ looks for $u_h \in (\Polybnd{\Degree}{1})^2$ and $p_h \in \Polyavg{\Degree-1}{1}$ such that
\begin{equation}
\label{Stokes-Hood-Taylor}
\begin{alignedat}{2}
&\forall v_h \in (\Polybnd{\Degree}{1})^2
&\qquad
\Visc \,\Formveldisc^\Ht(u_h, v_h)
-\int_\Domain p_h \Div v_h 
&= \left\langle  f , \Smt^\Ht v_h \right\rangle  \\
&\forall q_h \in \Polyavg{\Degree-1}{1}
&\qquad
\int_\Domain q_h \Div u_h &= 0.
\end{alignedat}
\end{equation}

This discretization is computationally feasible in the sense of
Remark~\ref{R:computational-feasibility},
cf. Remark~\ref{R:Pl-Pl-2-feasibility}. Yet, the implementation is
more costly than the one of~\eqref{Stokes-Pl-Pl-2}
and~\eqref{Stokes-CR} because, in general, we cannot resort to one
reference configuration for the solution of the local
problems~\eqref{local-problem-div-star}. The error analysis
of~\eqref{Stokes-Hood-Taylor} proceeds almost verbatim as in
section~\ref{SS:error-estimates}, with the help of
Proposition~\ref{P:Hood-Taylor-smoother}. The only remarkable
difference is that estimate~\eqref{Pl-Pl-2-pressure-error} in the
proof of Theorem~\ref{T:Pl-Pl-2-pressure-error} should be replaced by
the weaker one $\NormLeb{p_h-q_h}{} \Cleq \Visc \eta
\NormLeb{\Grad(u-u_h)}{} + \NormLeb{p-q_h}{}$, because
identity~\eqref{conservation-divergence} may fail to hold.

\section{Numerical experiments with the unbalanced $\Poly{2}/\Poly{0}$ pair} 
\label{S:numerics}

In this section we restrict our attention to the two-dimensional Stokes equations, with unit viscosity, posed in the unit square. In the notation of section~\ref{S:abstract-framework}, this corresponds to 
\begin{equation*}
\label{numerics-setting}
\Dim = 2 \qquad \qquad \Visc = 1 \qquad \qquad \Domain = (0,1)^2.
\end{equation*}
We investigate numerically the new discretization~\eqref{Stokes-Pl-Pl-2}, based on the unbalanced $\Poly{2}/\Poly{0}$ pair, i.e. 
\begin{equation*}
\label{P2-P0-pair}
\SPveldisc = (\Polybnd{2}{1})^2 
\qquad \text{and} \qquad
\SPpresdisc = \Polyavg{0}{0},
\qquad
\Formpresdisc(v_h, q_h) = -\int_\Domain q_h \Div v_h.
\end{equation*}
If not specified differently, the penalty parameter is set to 
\begin{equation*}
\label{numerics-penalty}
\eta = 2.
\end{equation*}

We shall consider the following families $(\Mesh_N^D)_{N \in \N_0}$ and $(\Mesh_N^C)_{N \in \N_0}$ of triangular meshes of $\Domain$. For $N \in \N_0$, we divide $\Domain$ into $2^N \times 2^N$ identical squares, with edges parallel to the $x_1$- and $x_2$-axis and with area $2^{-2N}$. We obtain the "diagonal mesh" $\Mesh_N^D$ dividing each square by the diagonal with positive slope. Similarly, we obtain the "crisscross mesh" $\Mesh_N^C$ drawing both diagonals of each square, cf. Figure~\ref{F:meshes}. All experiments have been implemented in ALBERTA 3.0~\cite{Heine.Koester.Kriessl.Schmidt.Siebert,Schmidt.Siebert:05}.
 
\begin{figure}[htp]
	\hfill
	\subfloat{\includegraphics[width=0.4\hsize]{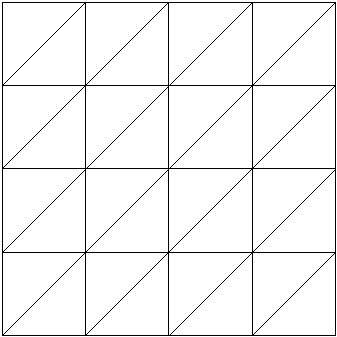}}
	\hfill
	\subfloat{\includegraphics[width=0.4\hsize]{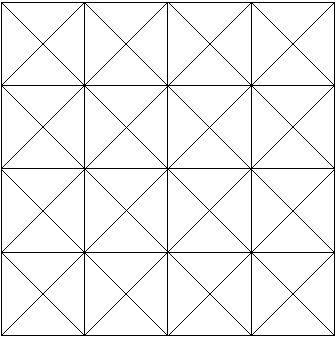}}
	\hfill
	\caption{Diagonal mesh $\Mesh_N^D$ (left) and crisscross mesh $\Mesh_N^C$ (right) with $N=2$.}
	\label{F:meshes}
\end{figure}

\subsection{Smooth solution}
\label{SS:numerics-smooth}

To illustrate the quasi-optimality and pressure robustness of the new $\Poly{2}/\Poly{0}$ discretization, we first consider a test case with smooth analytical solution, given by
\begin{equation*}
u(x_1, x_2) = \Curl(x_1^2(1-x_1)^2x_2^2(1-x_2)^2 )
\qquad
p(x_1, x_2) = \sin(2\pi x_1) \sin(2\pi x_2)
\end{equation*}
where $\Curl(w) := (\partial_2 w, -\partial_1 w)$. We compare the performances of the standard $\Poly{2}/\Poly{0}$ discretization~\eqref{Stokes-Pl-Pl-2-standard} and the new one~\eqref{Stokes-Pl-Pl-2} on the crisscross meshes $\Mesh_N^C$ with $N=0, \dots, 8$. Figure~\ref{F:smooth} displays the respective balances of velocity $H^1$-error and pressure $L^2$-error versus $\#\Mesh_N^C$, that is the number of triangles in the mesh. 

We first observe that the pressure $L^2$-errors of both discretizations behave quite similarly and converge to zero with the maximum decay rate $(\#\Mesh_N^C)^{-0.5}$. The velocity $H^1$-error of the standard discretization converges to zero with the same decay rate, as suggested by estimate~\eqref{intro:conforming-est}, according to the approximation power of the discrete pressure space in the $L^2$-norm. Note, however, that such rate is suboptimal with respect to the approximation power of the discrete velocity space in the $H^1$-norm. In contrast, the velocity $H^1$-error of the new discretization exhibits the maximum decay rate $(\#\Mesh_N^C)^{-1}$, as predicted by Theorem~\ref{T:Pl-Pl-2-velocity-error}. The next experiments are intended to highlight some of the ingredients that contribute to make this optimal-order convergence possible.

\begin{figure}[htp]
	\hfill
	\subfloat{\includegraphics[width=0.5\hsize]{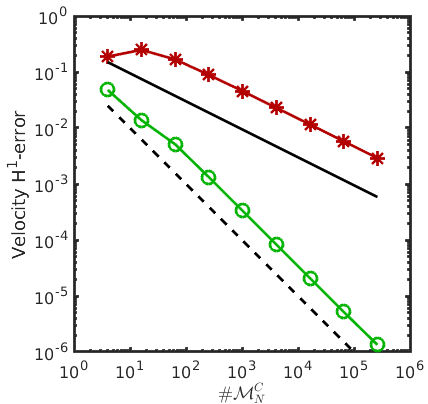}}
	\hfill
	\subfloat{\includegraphics[width=0.5\hsize]{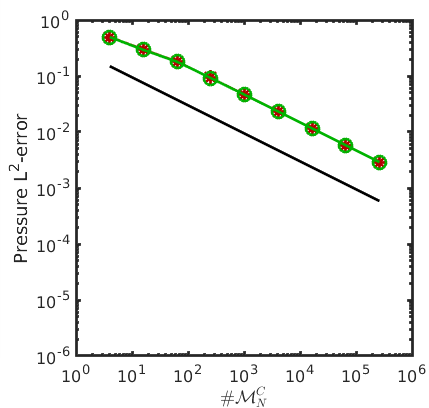}}
	\hfill
	\caption{Test case \S\ref{SS:numerics-smooth}. Velocity $H^1$-error (left) and pressure $L^2$-error (right) of standard ($*$) and new ($\circ$) $\Poly{2}/\Poly{0}$ discretizations. Plain and dashed lines indicate decay rates $(\#\Mesh_N^C)^{-0.5}$ and $(\#\Mesh_N^C)^{-1}$, respectively.}
	\label{F:smooth}
\end{figure}

\subsection{Composite numerical quadrature}
\label{SS:numerics-quadrature}

The evaluation of the duality $\left\langle f, \Smt v_h\right\rangle $, $v_h \in (\Polybnd{2}{1})^2$, in the new $\Poly{2}/\Poly{0}$ discretization requires, in particular, the evaluation of $\left\langle f, \widetilde{v}_h\right\rangle$ for test functions $\widetilde{v}_h$ that are element-wise quadratic on the barycentric refinement of the mesh at hand. This suggests that, for each triangle $K$ in the mesh, a composite quadrature rule, based on the barycentric refinement of $K$, should be used. If one, instead, uses a standard quadrature rule in $K$, the resulting quadrature error could be not negligible, due to the low regularity of $\widetilde{v}_h$. Moreover, since the quadrature error is potentially not pressure robust, as pointed out in~\cite[section ~6.2]{Linke.Merdon.Neilan.Neumann:18}, this may even affect the decay rate of the velocity $H^1$-error. 

To illustrate such effect, we consider a test case with analytical solution
\begin{equation*}
u(x_1, x_2) = \Curl(x_1^2(1-x_1)^2x_2^2(1-x_2)^2 )
\qquad
p(x_1, x_2) = \alpha\sin(2\pi x_1) \sin(2\pi x_2).
\end{equation*}
For $\alpha \in \{1, 10^3\}$, we apply the new $\Poly{2}/\Poly{0}$ discretization on the crisscross meshes $\Mesh_N^C$ with $N=0,\dots,8$. We assemble the right-hand side both with a composite and a standard quadrature rule of degree $6$. For $N=4,\dots, 8$, the corresponding velocity $H^1$-errors are reported in Table~\ref{F:quadrature}. In each case, we compute also the so-called experimental order of convergence (EOC), defined as 
\begin{equation*}
\label{EOC}
\mathrm{EOC}_N 
:= 
\frac{\log(e_N / e_{N-1})}{\log(\#\Mesh_{N-1}^C / \#\Mesh_N^C)} 
=
\frac{\log(e_{N-1} / e_N)}{\log 4}
\end{equation*}   
where $e_N$ denotes the $H^1$-error on $\Mesh_N^C$. 

When the composite quadrature rule is applied, the results seem
insensitive to the parameter $\alpha$ and we observe the maximum decay
rate $(\#\Mesh_N^C)^{-1}$. In contrast, the use of the standard
quadrature rule impairs the pressure robustness stated in
Theorem~\ref{T:Pl-Pl-2-velocity-error}. In fact, for sufficiently
large $N$, the velocity $H^1$-error is essentially proportional to
$\alpha$ and exhibits the suboptimal decay rate
$(\#\Mesh_N^C)^{-0.5}$.  

\begin{table}[htp]
	\begin{minipage}[c]{0.49\linewidth}
		\centering
		\begin{tabular}{|r|c|c|}
			&
			$\alpha = 1$ &
			$\alpha = 10^3$ \\
			N &
			$H^1$-error  \hspace{1pt} EOC &
			$H^1$-error  \hspace{1pt} EOC \\[1ex]
			\hline
			&&
			\\[-1.5ex]
			4 & 3.32e-04 \hspace{27pt} 
			  & 3.32e-04 \hspace{27pt} 
			\\
			5 & 8.31e-05 \hspace{5pt} \raisebox{1.5ex}[0pt]{1.00} 
			  & 8.31e-05 \hspace{5pt} \raisebox{1.5ex}[0pt]{1.00}
			\\
			6 & 2.08e-05 \hspace{5pt} \raisebox{1.5ex}[0pt]{1.00} 
			  & 2.08e-05 \hspace{5pt} \raisebox{1.5ex}[0pt]{1.00}
			\\
			7 & 5.19e-06 \hspace{5pt} \raisebox{1.5ex}[0pt]{1.00} 
			  & 5.19e-06 \hspace{5pt} \raisebox{1.5ex}[0pt]{1.00}
			\\
			8 & 1.30e-06 \hspace{5pt} \raisebox{1.5ex}[0pt]{1.00} 
			  &	1.30e-06 \hspace{5pt} \raisebox{1.5ex}[0pt]{1.00}
		\end{tabular}
	\end{minipage}
	\hfill
\begin{minipage}[c]{0.49\linewidth}
	\centering
	\begin{tabular}{|r|c|c|}
		&
		$\alpha = 1$ &
		$\alpha = 10^3$\\
		N &
		$H^1$-error  \hspace{1pt} EOC &
		$H^1$-error  \hspace{1pt} EOC \\[1ex]
		\hline
		&&
		\\[-1.5ex]
		4 & 3.57e-04 \hspace{27pt} 
		  & 1.29e-01 \hspace{27pt} 
		\\
		5 & 1.07e-04 \hspace{5pt} \raisebox{1.5ex}[0pt]{0.87} 
		  & 6.72e-02 \hspace{5pt} \raisebox{1.5ex}[0pt]{0.47}
		\\
		6 & 4.01e-05 \hspace{5pt} \raisebox{1.5ex}[0pt]{0.71} 
		  & 3.41e-02 \hspace{5pt} \raisebox{1.5ex}[0pt]{0.49}
		\\
		7 & 1.80e-05 \hspace{5pt} \raisebox{1.5ex}[0pt]{0.58} 
		  & 1.71e-02 \hspace{5pt} \raisebox{1.5ex}[0pt]{0.50}
		\\
		8 & 8.72e-06 \hspace{5pt} \raisebox{1.5ex}[0pt]{0.52} 
		  & 8.57e-03 \hspace{5pt} \raisebox{1.5ex}[0pt]{0.50}
	\end{tabular}
\end{minipage}
	\vspace{1ex}
	\caption{Test case \S\ref{SS:numerics-quadrature}. Velocity
          $H^1$-errors of the new $\Poly{2}/\Poly{0}$ discretization
          and corresponding EOCs with composite (left) or standard
          (right) quadrature rules for $\alpha \in \{1, 10^3\}$.} 
	\label{F:quadrature}
\end{table}

\subsection{Locking}
\label{SS:numerics-locking}

As mentioned in Remark~\ref{R:connection-DG}, the bilinear form
$\Formveldisc^\Pl$ in the new $\Poly{2}/\Poly{0}$ discretization has
the same structure as the DG-SIP form of~\cite{Arnold:82}. Still, one
main difference is that Lemma~\ref{L:Pl-Pl-2-coercivity} ensures the
coercivity of the former for any penalty $\eta>1$ (and not only for
sufficiently large $\eta$). Moreover, the coercivity constant is $\geq
0.5$ for $\eta = 2$. Having an explicit and safe choice of the penalty
parameter is particularly useful in this context, because we may have
locking for large $\eta$, in view of Remark~\ref{R:locking}. 

To illustrate this, we consider a test case with analytical solution 
\begin{equation*}
u(x_1, x_2) = \Curl(x_1^2(1-x_1)^2x_2^2(1-x_2)^2 )
\qquad
p(x_1, x_2) = (x_1 - 0.5)(x_2-0.5).
\end{equation*}
We apply the new $\Poly{2}/\Poly{0}$ discretization for $\eta \in \{
2, 32, 512 \}$ both on diagonal meshes $\Mesh_N^D$ and on crisscross
meshes $\Mesh_N^D$, with $N=0, \dots, 7$.  
%
The velocity $H^1$-errors displayed in the right part of
Figure~\ref{F:locking} indicate that the new discretization is robust
with respect to $\eta$ on crisscross meshes. This follows from the
fact that condition~\eqref{best-errors-Pl-Pl-2} in
Remark~\ref{R:locking} holds for such meshes, as a consequence
of~\cite[Theorem~4.3.1]{Qin:1994}. In contrast, adopting the
terminology 
of~\cite{Babuska.Suri:92b}, we observe on the left part of
Figure~\ref{F:locking} locking of order $(\Mesh_N^D)^{1/2}$ when
diagonal meshes are used.

\begin{figure}[htp]
	\hfill
	\subfloat{\includegraphics[width=0.5\hsize]{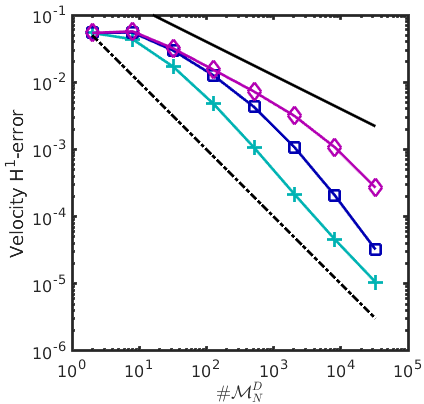}}
	\hfill
	\subfloat{\includegraphics[width=0.5\hsize]{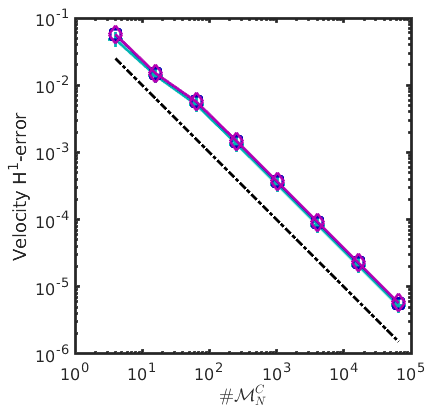}}
	\hfill
	\caption{Test case \S\ref{SS:numerics-locking}. Velocity
          $H^1$-error of the new $\Poly{2}/\Poly{0}$ discretization on
          diagonal (left) and crisscross (right) meshes, for $\eta =
          2$ ($+$), $\eta = 32$ ($\square$) and $\eta = 512$
          ($\Diamond$). Plain and dashed lines indicate decay rates
          $(\#\Mesh_N^*)^{-0.5}$ and $(\#\Mesh_N^*)^{-1}$, with $*\in
          \{D, C\}$.} 
	\label{F:locking}
\end{figure}

\subsection{Inhomogeneous continuity equation}
\label{SS:numerics-inhomogeneous}

We finally point out that the quasi-optimality and pressure robustness
of the new $\Poly{2}/\Poly{0}$ discretization, as stated in
Theorem~\ref{T:Pl-Pl-2-velocity-error}, hinges on the homogeneity of
the continuity equation in the Stokes problem~\eqref{Stokes-weak},
cf. section~\ref{SS:inhomogeneous-continuity}.  

\begin{figure}[htp]
	{\includegraphics[width=0.6\hsize]{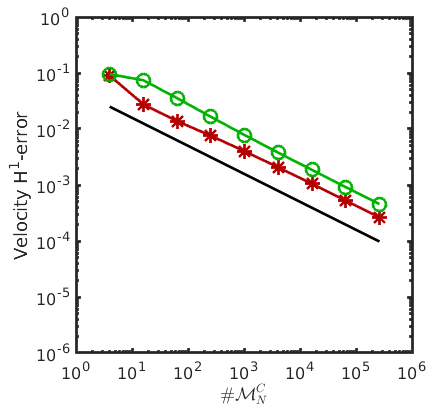}}
	\caption{Test case \S\ref{SS:numerics-inhomogeneous}. Velocity
          $H^1$-error of standard ($*$) and new ($\circ$)
          $\Poly{2}/\Poly{0}$ discretizations. Plain line indicates
          decay rate $(\#\Mesh_N^C)^{-0.5}$.}  
	\label{F:inhomogeneous}
\end{figure}

To see this, we consider the more general problem~\eqref{Stokes-inhom}
and approximate the analytical solution 
\begin{equation*}
u(x_1, x_2) = \left( 
\begin{tabular}{c}  
$x_1(1-x_1)x_2(1-x_2)$\\[2pt] $x_1(1-x_1)x_2(1-x_2)$
\end{tabular} \right) 
\qquad
p(x_1, x_2) = (x_1 - 0.5)(x_2-0.5)
\end{equation*}
on the crisscross meshes $\Mesh_N^C$ with $N=0,\dots,8$. Note, in
particular, that $\Div u$ is not element-wise constant on $\Mesh_N^C$.  

Comparing the velocity $H^1$-errors of the standard
$\Poly{2}/\Poly{0}$ discretization~\eqref{Stokes-Pl-Pl-2-standard} and
the new one~\eqref{Stokes-Pl-Pl-2}, we see that the former is slightly
smaller than the latter and that both errors converge to zero with
decay rate $(\Mesh_N^C)^{-0.5}$;
cf. Figure~\ref{F:inhomogeneous}. This confirms that
inequality~\eqref{Pl-Pl-2-velocity-error-inhom} captures the correct
behavior of the new discretization. Thus, for this problem, we expect
that the new discretization performs significantly better than the
standard one only in case of large pressure $L^2$-errors.

\subsection*{Acknowledgements}
We wish to thank Rüdiger Verfürth for reading some preliminary
versions of this manuscript and for suggesting several improvements in
the presentation. 

\subsection*{Funding}
The authors gratefully acknowledge
partial support by the DFG research grant KR 3984/5-1 ``Convergence
Analysis for Adaptive Discontinuous Galerkin Methods''.

\end{document}